\documentclass[12pt]{amsart}

\usepackage[hyperref,x11names]{xcolor}
\usepackage{hyperref}
\hypersetup{
colorlinks = true,       
breaklinks = true,       
urlcolor   = SteelBlue3, 
 linkcolor  = blue, 
citecolor = Firebrick3, 
}

\usepackage[T1]{fontenc} 
\usepackage[utf8]{inputenc}
\usepackage[english=british]{csquotes}
\usepackage{subfig}

\usepackage{lmodern}
\usepackage{todonotes}
\usepackage{
  amssymb, 
  amsmath, 
  amsfonts,
  amsthm, 
}

\usepackage{mathtools}
\usepackage{mathrsfs}
\usepackage[left=3cm, right=3cm, top=2.5cm,bottom=2.5cm]{geometry}

\usepackage[capitalise,noabbrev]{cleveref}

\usepackage{microtype}

\usepackage{tikz}       
\usepackage{pgfplots}
\usepackage{tikz-cd}
\usetikzlibrary{3d,shapes,calc,arrows,intersections} 
\usepackage{tikz-cd}

\tikzset{
  > = {Stealth},
  inner sep = 1.8pt,
  outer sep = auto,
}

\newtheorem{theorem}{Theorem}[section]
\newtheorem{lemma}[theorem]{Lemma}
\newtheorem{proposition}[theorem]{Proposition}

\newtheorem{corollary}[theorem]{Corollary}
\newtheorem{remark}[theorem]{Remark}
\theoremstyle{definition}
\newtheorem{definition}[theorem]{Definition}

\theoremstyle{remark} 
\newtheorem*{question*}{\sc Question}
\newtheorem*{example*}{\sc Example}

\newtheorem{alg}[theorem]{\sc Algorithm}

\numberwithin{equation}{section}

\usepackage{verbatim} 
\usepackage{framed} 
\usepackage{float} 
\usepackage{enumitem}

\newcommand*{\classicsets}[1]{\mathbb{#1}}

\newcommand*{\R}{\classicsets{R}}
\newcommand*{\bR}{\classicsets{R}}   
\newcommand*{\rgz}{\bR_{> 0}} 
\newcommand*{\rgez}{\bR_{\ge 0}} 
\newcommand*{\bP}{\mathbb{P}}
\newcommand*{\by}{\mathbf{y}}
\newcommand*{\byz}{\mathbf{y_0}}

\newcommand*{\vE}{{E}}

\newcommand*{\bxz}{\mathbf{x_0}}

\DeclareMathOperator{\dd}{d}

\usepackage{lineno}
\usepackage[backref=true,maxbibnames=99,natbib=true,backend=bibtex]{biblatex}
\title{Disguised Toric Dynamical Systems}
\date{\today}
\author{Laura Brustenga i Moncusí, Gheorghe Craciun, Miruna-\c Stefana Sorea}

\makeatletter
\@namedef{subjclassname@2020}{
  \textup{2020} Mathematics Subject Classification}
\makeatother

\begin{document}
\begin{abstract}
  \noindent We study families of polynomial dynamical systems inspired by biochemical reaction networks. We focus on \emph{complex balanced mass-action systems}, which have also been called \emph{toric}. They are known or conjectured to enjoy very strong dynamical properties, such as existence and uniqueness of positive steady states, local and global stability, persistence, and permanence.
  We consider the class of \emph{disguised toric dynamical systems}, which contains toric dynamical systems, and to which all dynamical properties mentioned above  extend naturally. By means of (real) algebraic geometry we show that some reaction networks have an {\em empty} toric locus or a toric locus of Lebesgue measure zero in parameter space, while their \emph{disguised} toric locus is of positive measure. We also propose \emph{some algorithms} one can use to detect the disguised toric locus.
\end{abstract}
\maketitle

\tableofcontents

\section*{Introduction}
Nonlinear dynamical systems are ubiquitous in the study of many natural phenomena (see, for instance, \cite{guckenheimer2013nonlinear, strogatz2001nonlinear, cy} and references therein). In particular, they have many important applications, in biology and medicine, such as in studying of the spread of infectious diseases, the dynamics of concentrations in biochemical reaction networks, or the dynamics of populations for species that interact in an eco-system.

Inspired by Poincaré~\cite{Poin}, mathematicians decided to look for {\em qualitative} aspects of nonlinear dynamics, since explicit solutions of nonlinear dynamical systems are usually impossible to calculate.  Of course, qualitative questions also prove to be very challenging. For instance, the second part of Hilbert's 16th problem (see \cite{MR1898209} and references therein for the state of the art) concerning polynomial differential equations in the real plane remains open after more than a century. In particular, an upper bound for the number of limit cycles is not known {\em even for quadratic} vector fields in the plane. Another very important feature of many nonlinear dynamical systems is {\em chaotic dynamics}~\cite{guckenheimer2013nonlinear}. The best known chaotic system is the Lorenz system: a quadratic dynamical system in ${\mathbb R}^3$ which is known to have chaotic solutions (the ``butterfly effect'', see \cite{lorenz}).

Nonlinear dynamical systems modelling interaction networks are usually systems of differential equations generated by {\em reaction networks}. The latter are seen as directed graphs living in the Euclidean space, called Euclidean embedded graphs (see \cite{MR3920470}). In the context of mass-action kinetics, the qualitative dynamical properties of these systems are strongly related to the combinatorics of the corresponding Euclidean embedded graph that generates the system. In addition, the right-hand-side of such systems is given by polynomials with real coefficients, giving rise to fruitful connections with the field of algebraic geometry (\cite{DF, MR3457596, MR2561288}).

Actually, the same dynamical system (i.e. the same polynomial right hand side), can be generated by several {\em distinct} reaction networks (see for instance \cite{cjy, CrPa, cy}). In other words, by studying several reaction networks that generate the same system, we might deduce important and useful dynamical behaviour, that would not have been accessible to us via the initial network. In our present work, which is a follow-up of \cite{cjy}, we will use this important property.

In this paper we focus on {\em complex balanced} dynamical systems, which have been introduced in the fundamental paper by Horn and Jackson (\cite{MR400923}) in 1972. Complex balanced systems form a large class of nonlinear dynamical systems for which a remarkable amount of information is known. For example, Horn and Jackson proved the existence and uniqueness of positive equilibria in each stoichiometric compatibility class. In other words, up to conservation laws, complex balanced dynamical systems have a unique positive steady state; moreover, these steady states are locally asymptotically stable. In the last decades, complex balanced dynamical systems have been proven or conjectured to enjoy exceptionally strong dynamical properties, such as global stability in each stoichiometric compatibility class,  impossibility of oscillations and chaotic dynamics, persistence, and permanence.

One of the current main open questions and motivation in the field of chemical reaction network theory is the Global Attractor Conjecture, which was stated in 1974 by Horn in \cite{horn74}. For a recently proposed proof see \cite{GlAttrConj}. The conjecture says that complex balanced mass-action dynamical systems
are globally stable within each positive stoichiometric compatibility class, that is, they have a globally attracting point (up to conservation laws). See \cite[Section 2.2]{MR3920470}. The global attractor conjecture has been proven under several hypotheses: in the case where the dimension of the stoichiometric compatibility class is lower or equal to three, and in all dimensions if the Euclidean embedded graph is connected and in all dimensions in the case of strongly endotactic networks. See \cite{cy} and references therein for the state of the art.

Complex balanced dynamical systems have been recently called {\em toric dynamical systems}~\cite{MR2561288}, due to their strong connections to combinatorial and computational algebraic geometry. If the parameters of the system verify certain algebraic conditions, then the corresponding system is complex balanced (i.e., toric). It was also shown that the moduli spaces of toric dynamical systems are toric varieties~\cite{MR2561288}. This is advantageous, since toric varieties (see \cite{bernd}) have particularly nice computational and combinatorial features. However, for most networks, the set in parameter space that gives rise to toric systems has Lebesgue measure zero. We will refer to this set of parameters as the {\em toric locus} of the network. Note that recently there has been an increasing interest on the study of the toric locus, see also \cite{Connected}.

Our main contribution is proving that the dynamical properties of toric (i.e., complex balanced) dynamical systems are true for a larger class of dynamical systems, that we call \emph{disguised toric dynamical systems}, see \cref{def:disguised-toric-dynamical-system-II}. Roughly speaking, we expand the toric locus from sets of Lebesgue measure zero in the positive orthant, to sets of positive measure. We also present an explicit algorithm (see Algorithm \ref{algorithmDisguised1}) that can be used systematically in order to find \emph{the disguised toric locus} (Definition \ref{not:khat}), given a reaction network.

Regarding the structure of the paper: in \cref{sec:preliminaries} and \cref{sec:disguisedtoric} we give the standard terminology. Next we introduce the notion of \emph{disguised toric dynamical systems}.

\cref{sec:complete3pts} deals with the dynamical system generated by the complete graph consisting of three vertices. Here the toric locus is a codimension-1 semialgebraic variety inside the positive orthant. However, the disguised toric locus is the whole space of positive rate constants.

In \cref{sec:quadrilat} we focus on the dynamical system generated by the graph with four collinear vertices, that we call ``the quadrilateral on a line''. In Theorem \ref{th:4gon}, we show that the parameter space can be decomposed in four chambers, and we study the corresponding dynamical behaviour in each of these chambers. In particular, there are three chambers where the dynamical systems are complex balanced for every parameter values. However, in the fourth chamber there are parameter values for which the dynamical system is not disguised toric. Furthermore we give a complete characterization of the complex balanced dynamical systems belonging to this special chamber. That is, we give necessary and sufficient conditions for the systems in this chamber to be disguised toric: see \cref{thr:quadrilater-line-main}. In the literature (see \cite{shiuTh}) it was known that if the parameters lay on a certain hypersurface (the Segre variety in this case), then the system generated by the quadrilateral on a line is complex balanced. Our main result in Theorem \ref{thr:quadrilater-line-main} states that all the parameters above this Segre variety give rise to complex balanced dynamical systems. That is, we pass from a set of Lebesgue measure zero (a hypersurface) to a set of positive Lebesgue measure. Next, in \cref{sec:globally-stable-not-disguised-toric} we study the multi-stationarity region in the parameter space, which is delimited by the zero locus of a certain discriminant (Figure \ref{fig:discriminant}). To this end, we use the notion of detailed balance and tools from real algebraic geometry, such as the discriminant of a real polynomial and Descartes' rule of signs.

In \cref{sec:rectange-E-graph} we consider a non-weakly reversible reaction network. Note that weak-reversibility is a necessary condition for obtaining a nonempty toric locus. However, as was shown in \cite{cjy}, the dynamical system generated by this network can be realized by other reaction networks, which might exhibit nicer combinatorial properties, such as weak reversibility. Using the latter together with algebraic tools such as quantifier elimination, we show that the disguised toric locus of the system turns out to be a set of positive measure in the parameter space.

\cref{sec:the-n-gon-in-a-line} is dedicated to the generalization of our study. We prove that the equilibria in the single-sign-change chambers for the ``$N$-gon on a line'' are detailed balanced, thus complex balanced.

In \cref{sect:algor2} we propose some systematic procedures one could follow in order to extend the toric locus to the \emph{disguised} toric locus of a polynomial dynamical system: this is Algorithm \ref{algorithmDisguised1}. The main tools in the algorithm are  properties of the Euclidean embedded graphs that generate the given dynamical system. The main idea is that one can add some degrees of freedom by introducing new positive real variables in the process of generating the same dynamical system using a different reaction network. Instances where we apply the steps of Algorithm \ref{algorithmDisguised1} appear throughout our paper: see the triangle on a line (\cref{sec:complete3pts}), the quadrilateral on a line (\cref{sec:quadrilat}) or \cref{sec:rectange-E-graph}, where by using \cref{algorithmDisguised1} we manage to extend an empty toric locus to a disguised toric locus of positive Lebesgue measure.

\section*{Acknowledgments}
The  authors would like to thank Bernd Sturmfels for bringing the team together, for giving us the opportunity to work on this project in the nice environment of the Nonlinear Algebra group at the Max Planck Institute for Mathematics in the Sciences, in Leipzig, and for his inspiring suggestions and comments. The authors would like to acknowledge the support of the Max Planck Institute for Mathematics in the Sciences, where most of this work was carried out. LBM’s contribution has been supported by the Novo Nordisk Foundation grant NNF18OC0052483. GC was supported by NSF grants DMS-1816238 and DMS-2051568 and by a Simons Foundation fellowship. MSS thanks Antonio Lerario and Andrei Agrachev for their support and excellent working conditions during her postdoc at SISSA, Trieste.

\section{Preliminaries}
\label{sec:preliminaries}

In this section we present standard terminology and notations for the study of chemical reaction networks. We refer to the textbook \cite{feinberg} for a complete introduction to the subject.

 Throughout this paper, by $\rgz^\vE$ we denote tuples of real numbers indexed by elements of $E$. We use bold letters to refer multi-index objects as vectors and monomials.

\begin{definition}\cite{MR3920470}\label{def:E-graph}
  A \emph{Euclidean embedded graph} (or \emph{E-graph} for short) is a
  digraph (directed graph) \(G=(V,E)\), where \(V\subseteq \bR^n\) is the set of vertices, \(E\) is the set of edges with no self-loops and at most one edge between a pair of ordered vertices.
  Given an edge \((\by,\by')\in E\) we also write \(\by\to \by'\in E\).
  Moreover, the vertices \(\by\), \(\by'\) are called respectively the \emph{source} and the \emph{product} of the edge \(\by\to \by'\).
\end{definition}

A reaction network can be regarded as an E-graph \(G=(V,E)\) where \(E\) is the set of reactions \cite{MR3920470}.
Sometimes, we will refer to an E-graph as \emph{a reaction network}, in order to emphasise its applied side.
To this end, we will also refer to vertices as \emph{complexes} and to edges as \emph{reactions}.
So, the restriction on \(E\) ensures that there is no reaction from a complex to itself (with no self-loops) and there is at most one reaction from one complex to another (at most one edge between a pair of ordered vertices).

\begin{definition}\label{def:system-assosatied-an-embgrph}
  Given an E-graph \(G=(V,E)\) with \(V\subseteq\bR^n\) and \(\mathbf{k}\in \rgz^\vE\), consider the function
  \begin{align}\label{eq:polynSyst}
    F_{G,\mathbf{k}}(\mathbf{x})\coloneqq \sum_{\by\to \by'\in E}k_{\by\to \by'}\mathbf{x}^{\by}(\by'-\by)
  \end{align} where, for \(\by=(\alpha_1,\dots,\alpha_n)\in\bR^n\), \(\mathbf{x}^\by\coloneqq x_1^{\alpha_1}\cdots x_n^{\alpha_n}\).
  The positive real number \(k_{\by\to \by'}\) is a \emph{rate constant} corresponding to the reaction \(\by\to \by'\) and the \emph{dynamical system generated by \(G\) and \(\mathbf{k}\in \rgz^\vE\)} is the following dynamical system.
  \begin{equation}\label{eq:dynSyst}
    \frac{\dd\mathbf{x}}{\dd t} = F_{G,\mathbf{k}}(\mathbf{x}).
  \end{equation}
\end{definition}

Note that, for \(V\subseteq\rgez^n\) \cref{def:system-assosatied-an-embgrph} corresponds to mass-action kinetics (\cite[page 28]{feinberg}).
For \(V\subseteq\mathbb{N}^n\), in the setting of mass-action kinetics, the function \(F_{G,\mathbf{k}}(\mathbf{x})\) gives rise to a \emph{polynomial} dynamical system in (\ref{eq:dynSyst}), which is the case for most practical applications.
Moreover, in this case, the set of source vertices corresponds to the set of monomials appearing in $F_{G,\mathbf{k}}(\mathbf{x})$.

Inspired by \cite{MR2800059} and \cite{cjy}, we give the following definition:
\begin{definition}\label{def:dynamic-inclusion}
  A particular dynamical system
  $$\frac{\dd\mathbf{x}}{\dd t} = f(\mathbf{x})$$
  \emph{has a realization using} an E-graph $G=(V,E)$ if there exists $\mathbf{k}\in\rgz^\vE$ with
  $$F_{G,\mathbf{k}}(\mathbf{x})=f(\mathbf{x}) \mbox{ for all } \mathbf{x}\in \rgez^n.$$
\end{definition}

\begin{definition}\label{def:same-dynamic-locus}
  Given an E-graph $G=(V,E)$ and $\mathbf{k}\in \rgz^E$, the \emph{$G$-equidynamic locus of} $\mathbf{k}$ is the set
\[
  \mathcal{V}(G,\mathbf{k}) := \{\mathbf{k}'\in\rgz^E \mbox{ : } F_{G,\mathbf{k}}(\mathbf{x})=F_{G,\mathbf{k}'}(\mathbf{x}) \mbox{ for all } \mathbf{x}\in\rgz^n\}. 
\] Given a subset $\Omega\subseteq \rgz^E$, the \emph{$G$-dynamic completion} of \(\Omega\) is the set
\[
  \mathcal{V}(G,\Omega) := \bigcup_{\mathbf{k}\in \Omega}\mathcal{V}(G,\mathbf{k}).
\]
\end{definition}
Fix an E-graph $G=(V,E)$. We describe the $G$-equidynamic locus for a $\mathbf{k}\in \rgz^E$.
Let $S=\{\mathbf{y}_1,\dots,\mathbf{y}_l\}$ be the set of sources of $G$ with some fixed order.
Let $E_i=\{\mathbf{y}_i\to \mathbf{y}\in E\}$ be the set of reaction whose source is $\mathbf{y}_i$ and set $n_i$ as the cardinal of $E_i$.
Consider the $n\times n_i$ matrix $A_i$ whose columns are the vectors $\mathbf{y}-\mathbf{y}_i\in E_i$ with some fixed order.
Given $\mathbf{k}\in\rgz^E$, consider the column vector $\mathbf{k}_i=(k_{\mathbf{y}_i\to \mathbf{y}})\in\rgz^{E_i}$ with the same order as the matrix $A_i$.
Now,
\[
  F_{G,\mathbf{k}}(\mathbf{x})= \sum_{i=1}^l \bigl(A_i\mathbf{k}_i\bigr)\mathbf{x}^{\mathbf{y}_i},
\] and it is straightforward to see that $\mathcal{V}(G,\mathbf{k})$ is the polyhedral cone
\[
  \mathcal{V}(G,\mathbf{k}) = \mathbf{k} + \ker A_1\times\dots\times\ker A_l,
\] where we consider $\R^E=\R^{E_1}\times\dots\times \R^{E_l}$ and $A_i$ as a linear map $\R^{E_i} \to \R^n$.

\begin{definition}
  Given an E-graph \(G=(V,E)\) with \(V\subseteq \bR^n\), the \emph{inflow} at the vertex $\byz\in V$ and the state $\bxz\in\rgez^n$ is the number
  \begin{align}
    \sum_{\by\to \byz\in E}k_{\by\to \byz}\bxz^{\by},
  \end{align} and the \emph{outflow} at $\byz$ and $\bxz$ is
  \begin{align}
    \biggl(\sum_{\byz\to \by\in E}k_{\byz\to \by}\biggr)\bxz^{\byz}.
  \end{align}
\end{definition}

In terms of reaction networks the inflow can be interpreted as the total production of the complex \(\byz\) per unit of time when the reaction network is at the state \(\bxz\).
Similarly, the outflow corresponds to the total rate at which the complex $\byz$ is being consumed.

\begin{definition}\cite{MR400923,cy,MR2561288}\label{def:toric-system}
  Given an E-graph \(G=(V,E)\) and \(\mathbf{k}\in\rgz^\vE\), the couple \((G,\mathbf{k})\) \emph{satisfies the complex balanced condition} if there exists \(\bxz\in\rgz^n\) satisfying the equation
  \begin{equation}\label{eq:complexBal}
    \biggl(\sum_{\byz\to \by\in E}k_{\byz\to \by}\biggr)\bxz^{\byz}=\sum_{\by\to \byz\in E}k_{\by\to \byz} \bxz^\by
  \end{equation} for every complex \(\byz\in V\).
  When such an $\bxz$ exists, the dynamical system (\ref{eq:dynSyst}) generated by \((G,\mathbf{k})\) is called a \emph{toric dynamical system} and $\bxz$ is called a \emph{complex balanced steady state} of (\ref{eq:dynSyst}).
\end{definition}

Recall that, when it exists, a complex balanced steady state is a steady state.
In fact, the existence of a single complex balanced steady state implies that every steady state is complex balanced~\cite{MR400923}.
Moreover, the steady states variety of a toric dynamical system is a toric variety~\cite{MR2561288}.

In terms of reaction networks the complex balanced condition has a very clear interpretation.
It is asking for the existence of a state \(\bxz\) for which the inflow and the outflow are equal (or, in other words, are balanced) at every complex.
That is, if a complex balanced steady state \(\bxz\) exists, when the reaction network is at such state the total production of each complex is balanced with its total outflow.

The first to study toric dynamical systems, Horn and Jackson~\cite{MR400923}, called these systems \emph{complex balanced dynamical systems}.
The new terminology has been  introduced in \cite{MR2561288}, where the authors studied these systems from the computational algebraic geometry point of view.
Given an E-graph $G$, they showed that the space of parameters $\mathbf{k}$ for which the couple $(G,\mathbf{k})$ satisfies the complex balance condition is a toric variety.
It is well known that toric ideals (see \cite{bernd}) have particularly nice algebraic and combinatorial properties: \enquote{the world is toric} (\cite[Section 8.3]{BM}).

\begin{remark}\label{rmk:toric-imply-wreaversible1}
  A necessary condition for a dynamical system to be complex balanced is the weak reversibility of the graph that generates it~\cite{MR400923, feinberg}.
\end{remark}

\begin{definition}\cite{cy,feinberg}\label{def:detailed-system}
  Given a reversible E-graph \(G=(V,E)\) and \(\mathbf{k}\in\rgz^\vE\), the couple \((G,\mathbf{k})\) \emph{satisfies the detailed balanced condition} if there exists \(\bxz\in\rgz^n\) that satisfies
  \begin{equation}
    k_{\byz\to \by}\bxz^{\byz}=k_{\by\to \byz} \bxz^\by
  \end{equation} for every reaction $\by\to \byz$ in $E$.
  When such an \(\bxz\) exists, it is called a \emph{detailed balanced steady state}.
\end{definition}

\begin{definition}\cite[Definition 2.7]{cy}\label{def:deficiency}
  Given an E-graph $G$, denote by $n$ the number of vertices, by $l$ the number of its connected components, and by $s$ the dimension of the stoichiometric compatibility class (that is, the vector subspace generated by the edges of $G$). Then the \emph{deficiency} of $G$ is the integer $\delta\coloneqq n-l-s$.
\end{definition}

\section{Disguised Toric Dynamical Systems}\label{sec:disguisedtoric}

In this section we introduce the main objects of study: disguised toric dynamical systems and the disguised toric locus of an E-graph.

\begin{definition}\label{def:disguised-toric-dynamical-system-II}
  Given a dynamical system
  \begin{equation}\label{eq:def:system-in-a-graph}
    \frac{\dd\mathbf{x}}{\dd t}=f(\mathbf{x}) \mbox{ on } \mathbf{x}\in\rgez^n,
  \end{equation}
  we say that it is a \emph{disguised toric dynamical system} if there exist an E-graph $G=(V,E)$ and \(\mathbf{k}\in\rgz^\vE\) such that
  \begin{equation*}
    f(\mathbf{x})=F_{G,\mathbf{k}}(\mathbf{x}) \mbox{ for all } \mathbf{x}\in\rgez^n
  \end{equation*} and the couple \((G,\mathbf{k})\) satisfies the complex balanced condition.
  When~\eqref{eq:def:system-in-a-graph} is a disguised toric dynamical system, we also say that it \emph{has a complex balanced realization using the graph} $G$.
  \medskip

  We say that the dynamical system (\ref{eq:def:system-in-a-graph}) \emph{has a detailed balanced realization} if there exist an E-graph $G=(V,E)$ and \(\mathbf{k}\in\rgz^\vE\) such that
  \[
    f(\mathbf{x})=F_{G,\mathbf{k}}(\mathbf{x}) \mbox{ for all } \mathbf{x}\in\rgez^n
  \] and the couple \((G,\mathbf{k})\) satisfies the detailed balanced condition.
\end{definition}

\begin{definition}\label{not:khat}
  Given an E-graph $G=(V,E)$, we define respectively the \emph{toric locus} and the \emph{disguised toric locus} of \(G\) as the sets
  \[
    K(G)\coloneqq \{\mathbf{k}\in \rgz^\vE\mid \text{ the system generated by } (G,\mathbf{k}) \text{ is toric}\},
  \]
  \[
    \hat{K}(G)\coloneqq \{\mathbf{k}\in\rgz^\vE\mid \text{ the system generated by } (G,\mathbf{k}) \text{ is disguised toric} \}.
  \]
\end{definition}

\subsection{Goals and motivation}
Given an E-graph $G=(V,E)$, we are mostly interested in its toric and disguised toric locus.

A first trivial observation is that for every E-graph, \(K(G)\subseteq \hat{K}(G)\), since every toric dynamical system is disguised toric.

According to \cite[Theorem 9]{MR2561288}, if $G$ is weakly reversible, then $K(G)\subseteq \rgz^\vE$ is a semialgebraic variety of codimension $\delta$, where $\delta$ is the deficiency of $G$ (see \cref{def:deficiency}).

In particular, for every weakly reversible E-graphs with 0 deficiency, \(K(G)=\hat{K}(G)=\rgz^\vE\) (see \cite[Section 7.7]{feinberg}).

Most reaction networks (or E-graphs) coming from practical applications have positive deficiency.
So, for these reaction networks the toric locus is usually of measure zero in the space of rate constants.
That is the main problem with toric dynamical systems, they enjoy extremely pleasant dynamical properties but the chances for a particular dynamical system to be toric are really small.
But from the point of view of dynamical systems, the dynamical systems generated by a \(\mathbf{k}\) in the toric locus \(K(G)\) or in the disguised toric locus \(\hat{K}(G)\) are completely equivalent and, as we will show, $\hat{K}(G)$ is a much larger set in many cases.
For example, $\hat{K}(G)$ may have positive measure even when \(K(G)\) is empty (see \cref{sec:rectange-E-graph}).
\medskip

The disguised toric locus of an E-graph has been indirectly introduced in \cite{cjy}, where the authors used analytical methods to study it.
Our main goal here is, first to establish an explicit definition of disguised toric dynamical systems and disguised toric locus, and second to compute and to find approximations of  $\hat{K}(G)$ by means of algebraic methods.

\section{Triangle on a line}\label{sec:complete3pts}
In this section, we study the E-graph \(G=(V,E)\) given by~\cref{fig:complete3}.
It is the complete graph over the three vertices \(\by_1,\by_2,\by_3\).

This is an example where \(K(G)\subseteq\rgz^\vE\) is a codimension-1 semialgebraic variety, but \(\hat{K}(G)\) is the whole space of rate constants \(\rgz^\vE\).
This example also outlines a procedure to determine \(\hat{K}(G)\) where non explicit complex balanced realization is given.
More precisely, we will show that for every \(\mathbf{k}\in\rgz^\vE\), there exists \(\hat{\mathbf{k}}\in\rgz^\vE\) such that \((G,\hat{\mathbf{k}})\) satisfies the complex balanced condition and the systems generated by \(G\), \(\mathbf{k}\) and \(G\), \(\hat{\mathbf{k}}\) are equal.
That is, this example shows that sometimes the system generated by some rate constants can have a complex balanced realization using \emph{the very same graph} but different rate constants.
In other words, we show that $\mathcal{V}(G,\mathbf{k})\cap K(G)\not=\emptyset$ for all $\mathbf{k}\in \rgz^E$.

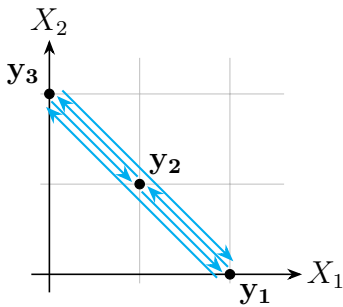
\begin{figure}
  \centering
  \begin{tikzpicture}[scale=1.2]
    \draw[ultra thin,lightgray](-0.1,-0.1) grid (2.6,2.4);
    \draw[-{Stealth},semithick] (-0.2,0) -- (2.8,0) node[right] {$X_1$};
    \draw[-{Stealth},semithick] (0,-0.2) -- (0,2.6) node[above] {$X_2$};
    \node[label=below right:{$\mathbf{y_1}$}](A) at (2,0) {};
    \node[label=above right:{$\mathbf{y_2}$}](B) at (1,1) {};
    \node[label=above left:{$\mathbf{y_3}$}](C) at (0,2) {};
    \draw[fill] (A) circle (1.5pt);
    \draw[fill] (B) circle (1.5pt);
    \draw[fill] (C) circle (1.5pt);

    \def\shf{0.2ex};
    \draw[->,transform canvas={xshift=\shf,yshift=\shf}, thick, cyan] (A) -- (B);
    \draw[->,transform canvas={xshift=-\shf,yshift=-\shf}, thick, cyan] (B) -- (A);
    \draw[->,transform canvas={xshift=\shf,yshift=\shf}, thick, cyan] (B) -- (C);
    \draw[->,transform canvas={xshift=-\shf,yshift=-\shf}, thick, cyan] (C) -- (B);

    \def\shf{0.6ex};
    \draw[->,transform canvas={xshift=\shf,yshift=\shf},thick, cyan] (C) -- (A);
    \draw[->,transform canvas={xshift=-\shf,yshift=-\shf},thick, cyan] (A) -- (C);
  \end{tikzpicture}
  \caption{Triangle on a line.}
  \label{fig:complete3}
\end{figure}
Given a vector of rate constants \(\mathbf{k}\in\rgz^\vE\), we simplify the notation \(k_{\by_i\to\by_j}\) to \(k_{ij}\).

Following the combinatorial scheme established in \cite{MR2561288}, the equation on \(k_{ij}\) for the toric locus \(K(G)\subseteq\rgz^\vE\) is
\begin{equation}\label{eq:paper2009}
  K_1K_3-K_2^2,
\end{equation} where \(K_i\) are the maximal minors of the negative of the Laplacian of the graph \(G\), which can be computed by means of the matrix-tree theorem.
Namely,
\begin{align*}
  K_1\coloneqq k_{21} k_{31}+k_{32} k_{21}+k_{23} k_{31};\\
  K_2\coloneqq k_{13} k_{23}+k_{21} k_{13}+k_{12} k_{23};\\
  K_3\coloneqq k_{12} k_{32}+k_{13} k_{32}+k_{31} k_{12}.
\end{align*}
Observe that equation (\ref{eq:paper2009}) defines a \emph{toric} variety in \(\bP^2\). This is a general fact for E-graphs proved in \cite{MR2561288}.
In fact, this toric variety is the rational normal curve in \(\bP^2\), and this is a general fact for strongly connected E-graphs contained in a hyperplane \(X_1+\dots+X_n=N\), see \cite[Proposition 5.2.1.]{shiuTh}

\begin{theorem}\label{prop:all}
  For the E-graph \(G=(V,E)\) given by~\cref{fig:complete3}, the disguised toric locus \(\hat{K}(G)\) is the whole space of rate constants \(\rgz^\vE\).
\end{theorem}

Before proving~\cref{prop:all} we reduce the E-graph \(G\) to have only one reaction per source. Observe that given \(\mathbf{k}\in\rgz^\vE\) we may realize the system generated by \((G,\mathbf{k})\) by a cycle directed graph over \(\by_1,\by_2,\by_3\) (see~\cref{fig:cycle-trigle-line}).
Indeed, consider the vectors
\[
  \mathbf{u}_i \coloneqq  \sum_{\by_i\to \by_j\in E}k_{ij}(\by_j-\by_i).
\]
Since \(k_{12}(\by_2-\by_1)\) and \(k_{13}(\by_3-\by_1)\) are positively proportional to \(\begin{psmallmatrix*}[r] -2 \\ 2 \end{psmallmatrix*}\), so is the vector \(\mathbf{u}_1\).
Let us denote by \(k^{*}_1\in\rgz\) the proportional factor, that is
\[
  \mathbf{u}_1= k^{*}_1\begin{pmatrix*}[r] -2 \\ 2 \end{pmatrix*}.
\] Similarly, let us denote by \(k^{*}_3\in\rgz\) the proportional factor between \(\mathbf{u}_3\) and \(\begin{psmallmatrix*}[r] 1 \\ -1 \end{psmallmatrix*}\).

The situation for the vector \(\mathbf{u}_2\) is slightly different.
Let us assume it is different from zero.
Depending on the values of \(k_{2j}\), it will be \emph{positively} proportional to \(\begin{psmallmatrix*}[r] -1 \\ 1 \end{psmallmatrix*}\) or to \(\begin{psmallmatrix*}[r] 1 \\ -1 \end{psmallmatrix*}\).
We assume that \(\mathbf{u}_2\) is positively proportional to the later, the other case simply corresponds to a permutation of the formal variables \(X_1\), \(X_2\).
So, let us denote by  \(k^{*}_2\in\rgz\) the proportional factor.

Now, consider the E-graph \(G^{*}=(V,E^{*})\) given by~\cref{fig:cycle-trigle-line}.
Finally, by construction, the dynamical systems generated by \(G\), \(\mathbf{k}\) and by \(G^{*}\), \(\mathbf{k}^{*}=(k^{*}_1,k^{*}_2,k^{*}_3)\in\rgz^{E^{*}}\) are equal.

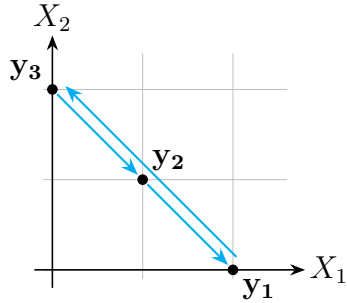
\begin{figure}[H]
  \centering
  \begin{tikzpicture}[scale=1.2]
    \draw[ultra thin,lightgray](-0.1,-0.1) grid (2.6,2.4);
    \draw[-{Stealth},semithick] (-0.2,0) -- (2.8,0) node[right] {$X_1$};
    \draw[-{Stealth},semithick] (0,-0.2) -- (0,2.6) node[above] {$X_2$};
    \node[label=below right:{$\mathbf{y_1}$}](A) at (2,0) {};
    \node[label=above right:{$\mathbf{y_2}$}](B) at (1,1) {};
    \node[label=above left:{$\mathbf{y_3}$}](C) at (0,2) {};
    \draw[fill] (A) circle (1.5pt);
    \draw[fill] (B) circle (1.5pt);
    \draw[fill] (C) circle (1.5pt);
    \def\shf{0.6ex};
    \draw[->,transform canvas={xshift=\shf,yshift=\shf},thick, cyan] (A) -- (C);
    \def\shf{0.2ex};
    \draw[->, thick, cyan] (C) -- (B);
    \draw[->, thick, cyan] (B) -- (A);
  \end{tikzpicture}
  \caption{Cycle on three vertices $G^*$.}
  \label{fig:cycle-trigle-line}
\end{figure}

\begin{proof}[Proof of~\cref{prop:all}]
  Clearly, we just need to prove that the disguised toric locus \(\hat{K}(G^{*})\) is the whole \(\rgz^{E^{*}}=\rgz^3\).

  Now, we will come back to the E-graph \(G\).
  That is, we realize the system generated by \(G^{*}\), \(\mathbf{k}^{*}\) using the graph \(G\).
  Fix \(\mathbf{k}^{*}\) and consider the E-graph \(\hat{G}=G\) and the rate constants \(\hat{\mathbf{k}}\) given by
  \begin{align}\label{eq:embedding3vertices}
    \begin{aligned}
      \hat{k}_{32} & \coloneqq \frac{1}{1+a}k_3^*\\
      \hat{k}_{31} & \coloneqq  \frac{a}{2(1+a)}k_3^*
    \end{aligned} & &
                      \begin{aligned}
                        \hat{k}_{21} & \coloneqq k_2^*+b\\
                        \hat{k}_{23} & \coloneqq  b\\
                      \end{aligned} & &
                                        \begin{aligned}
                                          \hat{k}_{12} & \coloneqq  2\frac{c}{1+c}k_1^*\\
                                          \hat{k}_{13} & \coloneqq \frac{1}{1+c}k_1^*
                                        \end{aligned}
  \end{align}
  where $a,b,c>0$.
  Equations~(\ref{eq:embedding3vertices}) are chosen so that the dynamical systems generated by \(G^{*}\), \(\mathbf{k}^{*}\) and \(\hat{G}\), \(\hat{\mathbf{k}}\) are equal for all \(a,b,c>0\).

  The toric locus for \(\hat{G}\) is given by equation~\eqref{eq:paper2009} substituting \(k_{ij}\) by \(\hat{k}_{ij}\).
  So, the pullback of the equation defining \(K(\hat{G})\subseteq\rgz^{\hat{E}}\) by~\eqref{eq:embedding3vertices} is the following function \(\varphi(a,b,c)\) on \(a,b,c>0\):
  \begin{align}\label{eq:subst}
    \begin{split}
      \left(\frac{k_{3}^* (b+k_{2}^*)}{a+1}+\frac{(a k_{3}^*) (b+k_{2}^*)}{2
          (a+1)}+\frac{b (a k_{3}^*)}{2 (a+1)}\right) \left(\frac{k_{1}^*
          (b+k_{2}^*)}{c+1}+\frac{b k_{1}^*}{c+1}+\frac{b (2 c
          k_{1}^*)}{c+1}\right)-\\ -\left(\frac{k_{1}^* k_{3}^*}{(a+1)
          (c+1)}+\frac{k_{3}^* (2 c k_{1}^*)}{(a+1) (c+1)}+\frac{(a k_{3}^*) (2 c
          k_{1}^*)}{(2 (a+1)) (c+1)}\right)^2.
    \end{split}
  \end{align}
  Now, if there exists \(a_0,b_0,c_0>0\) such that \(\varphi(a_0,b_0,c_0)=0\), then~\eqref{eq:embedding3vertices} for \(a_0,b_0,c_0\) will give a complex balanced realization of the system generated by \(G^{*}\), \(\mathbf{k}^{*}\) using the graph \(\hat{G}\).
  So to finish, we show that such \(a_0,b_0,c_0\) always exist regardless of \(\mathbf{k}^{*}\).

  First observe that, taking \(b_1\) large enough (tending to infinity), there exist real numbers $(a_1,b_1,c_1)$ with $a_1,b_1,c_1>0$ such that $\varphi(a_1,b_1,c_1)>0$ (the first term of \(\varphi\) grows to infinity and the second term is bounded).

  Second, taking \(b_2\) small enough (tending to zero) and $c_2$ large enough (tending to infinity), there exist $(a_2,b_2,c_2)$ with $a_2,b_2,c_2>0$ such that $\varphi(a_2,b_2,c_2)<0.$

  Hence, by the intermediate value theorem, there exist $(a,b,c)$ such that $\varphi(a,b,c)=0$, since the function $\varphi$ is continuous and its domain is connected.
\end{proof}

\begin{corollary}\label{coro:triagle-endotactic}
  For every endotactic (see \cite{CNP} for the definition) E-graph \(G'=(V',E')\) with the same three source vertices \(\by_1,\by_2,\by_3\) and for every \(\mathbf{k}'\in\rgz^{E'}\), the generated dynamical system is disguised toric.
\end{corollary}
\begin{proof}
  We can always realize the system generated by \(G'\), \(\mathbf{k}'\) using the E-graph \(G^{*}\) given by Figure~\ref{fig:cycle-trigle-line}.
  That is, for every $\mathbf{k}'\in\rgz^{E'}$, there are $\mathbf{k}^{*}\in\rgz^{E^{*}}$ (which we have proved that are disguised toric) generating the same system.
\end{proof}

\begin{remark}
  The same result (with essentially the same proof) remains true for  every weakly reversible E-graph with three vertices and deficiency one.
  More generally, the same is true for every endotactic network with the three source vertices and all vertices on a line.
\end{remark}

\section{Quadrilateral on a line}\label{sec:quadrilat}
In this section, we completely determine the dynamics of the systems generated by the E-graph \(G=(V,E)\) given by~\cref{fig:4thGhat}.
It is the complete graph over the four vertices \(\by_1,\by_2,\by_3,\by_4\).

This example outlines a procedure to find sufficient semialgebraic conditions on \(\mathbf{k}\in\rgz^\vE\) for being in \(\hat{K}(G)\).
The procedure could be described as follows.
For a given \(\mathbf{k}\in\rgz^\vE\), realize the dynamical system generated by \((G,\mathbf{k})\) using an E-graph \(\hat{G}=(\hat{V},\hat{E})\) where the detailed balance condition can be established.
Then, pullback to \(\mathbf{k}\in\rgz^\vE\) the equations of the detailed balance condition on \(\hat{\mathbf{k}}\in\rgz^{\hat{E}}\).
So, the obtained semialgebraic set will be contained in \(\hat{K}(G)\), since detailed balance dynamical systems are toric.

In fact, the previous procedure completely determines the disguised toric locus \(\hat{K}(G)\) in many cases, for this example and in~\cref{sec:the-n-gon-in-a-line} for all the so called single-sign-change chambers.
Nevertheless, in~\cref{sec:rectange-E-graph}, we will show that it may fail.

\begin{figure}
  \centering
  \begin{tikzpicture}[scale=1.2]
    \draw[ultra thin,lightgray](-0.1,-0.1) grid (3.6,3.4);
    \draw[-{Stealth},semithick] (-0.2,0) -- (3.8,0) node[right] {$X_1$};
    \draw[-{Stealth},semithick] (0,-0.2) -- (0,3.6) node[above] {$X_2$};
    \node[label=below right:{$\mathbf{y_1}$}](A) at (3,0) {};
    \node[label=above right:{$\mathbf{y_2}$}](B) at (2,1) {};
    \node[label=above right:{$\mathbf{y_3}$}](C) at (1,2) {};
    \node[label=above left:{$\mathbf{y_4}$}](D) at (0,3) {};
    \draw[fill] (A) circle (1.5pt);
    \draw[fill] (B) circle (1.5pt);
    \draw[fill] (C) circle (1.5pt);
    \draw[fill] (D) circle (1.5pt);
    \draw[<->,thick,cyan] (A) -- (B);
    \draw[<->,thick,cyan] (B) -- (C);
    \draw[<->,thick,cyan] (C) -- (D);
    \def\shf{0.5ex};
    \draw[<->,transform canvas={xshift=\shf,yshift=\shf},thick, cyan] (A) -- (C);
    \draw[<->,transform canvas={xshift=-\shf,yshift=-\shf},thick, cyan] (B) -- (D);
    \def\shf{0.9ex};
    \draw[<->,transform canvas={xshift=\shf,yshift=\shf},thick, cyan] (A) -- (D);
  \end{tikzpicture}
  \caption{Complete directed graph on 4 vertices.}
  \label{fig:4thGhat}
\end{figure}
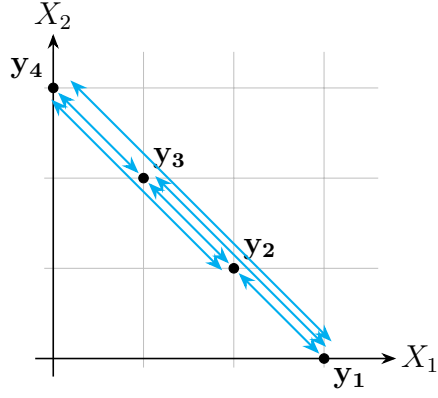

Similarly to~\cref{sec:complete3pts}, we start reducing the E-graph \(G\) to an E-graph \(G^{*}=(V,E^{*})\) with one reaction per vertex.
We do it systematically and aiming that the procedure will be easy to extrapolate to the more general E-graph of \cref{sec:the-n-gon-in-a-line}, the N-gone on a line.

For a given vector of rate constants \(\mathbf{k}\in\rgz^\vE\), let us simplify \(k_{\by_i\to\by_j}\) to \(k_{ij}\) and  consider the vectors
\[
  \mathbf{u}_i \coloneqq  \sum_{\by_i\to \by_j\in E}k_{ij}(\by_j-\by_i).
\] The system generated by \(G\), \(\mathbf{k}\) is
\begin{align*}\label{eq:Complete4_red}
  \begin{split}
    \frac{d}{dt} \begin{pmatrix*}[r]  x_1 \\ x_2 \end{pmatrix*} = &
    \,\,\mathbf{u}_1x_1^3+\mathbf{u}_2x_1^2x_2+\mathbf{u}_3x_1x_2^2+\mathbf{u}_4x_2^3 = \\
    = &\,\, (k_{12}+2k_{13}+3k_{14})\begin{psmallmatrix*}[r]  -1 \\ 1 \end{psmallmatrix*}x_1^3 +\\
    &+(k_{21}-k_{23}-2k_{24})\begin{psmallmatrix*}[r]  1 \\ -1 \end{psmallmatrix*}x_1^2 x_2+\\
    &+(2k_{31}+k_{32}-k_{34})\begin{psmallmatrix*}[r]  1 \\ -1 \end{psmallmatrix*}x_1 x_2^2+\\
    &+(3k_{41}+k_{43}+2k_{42})\begin{psmallmatrix*}[r]  1 \\ -1 \end{psmallmatrix*}x_2^3.
  \end{split}
\end{align*}

The vectors \(\mathbf{u}_1\), \(\mathbf{u}_4\) are respectively \emph{positively} proportional to \(\begin{psmallmatrix*}[r] -1 \\ 1 \end{psmallmatrix*}\) and \(\begin{psmallmatrix*}[r] 1 \\ -1 \end{psmallmatrix*}\).
Let us denote by \(k_1^{*},k_4^{*}\in\rgz\) the respective proportional factors, namely
\begin{align*}
  k_1^{*} & \coloneqq  k_{12}+2k_{13}+3k_{14};\\
  k_4^{*} & \coloneqq  3k_{41}+2k_{42}+k_{43}.
\end{align*}
Again, for the vectors \(\mathbf{u}_2\), \(\mathbf{u}_3\) the situation is slightly different, since their direction depends on the particular values of \(\mathbf{k}\in\rgz^\vE\).
Both vectors are always positively proportional to either \(\begin{psmallmatrix*}[r] -1 \\ 1 \end{psmallmatrix*}\) or \(\begin{psmallmatrix*}[r] 1 \\ -1 \end{psmallmatrix*}\).
We denote respectively the positive proportional factor for each case by \(k_2^{*}\) and \(k_3^{*}\), namely
\[
  k_2^{*}\coloneqq
  \begin{cases}
    k_{21}-k_{23}-2k_{24} & \mbox{ if } k_{21}-k_{23}-2k_{24}>0 \\
    -k_{21}+k_{23}+2k_{24} & \mbox{ otherwise }
  \end{cases}
\]
\[
  k_3^{*}\coloneqq
  \begin{cases}
    2k_{31}+k_{32}-k_{34} & \mbox{ if } 2k_{31}+k_{32}-k_{34}>0 \\
    -2k_{31}-k_{32}+k_{34} & \mbox{ otherwise. }
  \end{cases}
\]
So now, we have to distinguish four possible cases in order to realize the system generated by \(G\), \(\mathbf{k}\) by an E-graph with a reaction per source.
This cases are summarised in~\cref{th:4gon} below.

\begin{proposition}\label{th:4gon}
  Consider the E-graph \(G=(V,E)\) given by~\cref{fig:4thGhat}.
  Consider a vector of rate constants \(\mathbf{k}\in\rgz^\vE\) and consider \(\mathbf{k}^{*}=(k_1^{*},\dots,k_4^{*})\) defined above.
  The dynamical system generated by \(G\), \(\mathbf{k}\) is equal to the system generated by \(\mathbf{k}^{*}\) and
  \begin{enumerate}
    \item the E-graph given by (A) in~\cref{fig:chamb1234} if
          \[
          \begin{cases}
            k_{21}-k_{23}-2k_{24}\le 0; \\
            2k_{31}+k_{32}-k_{34}\le 0;
          \end{cases}
          \]
    \item the E-graph given by (B) in~\cref{fig:chamb1234} if
          \[
          \begin{cases}
            k_{21}-k_{23}-2k_{24}\ge 0;\\
            2k_{31}+k_{32}-k_{34}\le 0;
          \end{cases}
          \]
    \item the E-graph given by (C) in~\cref{fig:chamb1234} if
          \[
          \begin{cases}
            k_{21}-k_{23}-2k_{24}\ge 0;\\
            2k_{31}+k_{32}-k_{34}\ge 0;
          \end{cases}
          \]
    \item the E-graph given by (D) in~\cref{fig:chamb1234} if
          \[
          \begin{cases}
            k_{21}-k_{23}-2k_{24}\le 0;\\
            2k_{31}+k_{32}-k_{34}\ge 0.
          \end{cases}
          \]
  \end{enumerate}
\end{proposition}

\def\scopesp{7.5}
\def\dirlength{0.3cm}
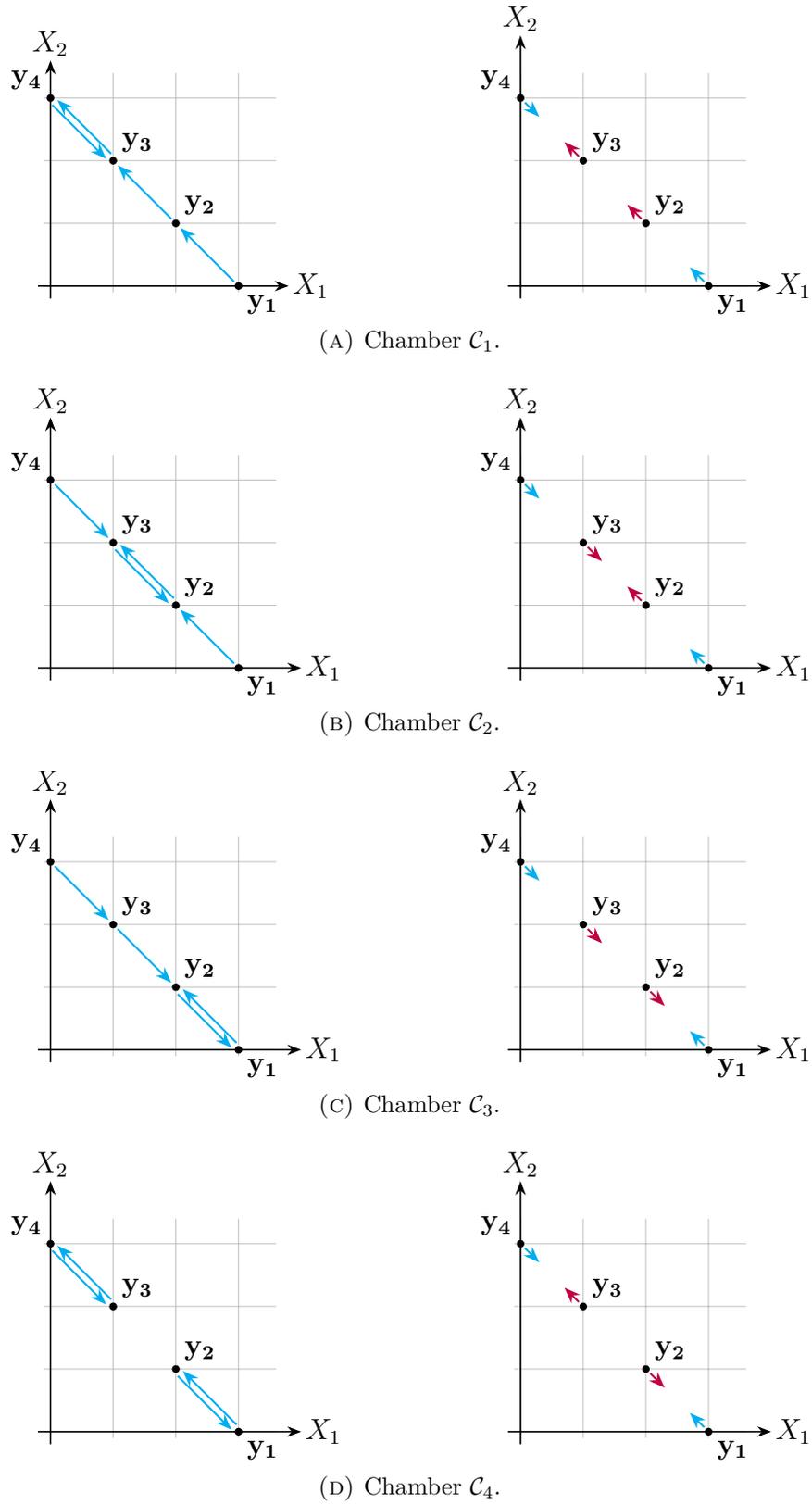
\begin{figure}
  \subfloat[Chamber $\mathcal{C}_1$.]{
    \begin{tikzpicture}[scale=0.9]
      \draw[ultra thin,lightgray] (-0.1,-0.1) grid (3.6,3.4);
      \draw[-{Stealth},semithick] (-0.2,0) -- (3.8,0) node[right] {$X_1$};
      \draw[-{Stealth},semithick] (0,-0.2) -- (0,3.6) node[above] {$X_2$};
      \node[label=below right:{$\mathbf{y_1}$}](A) at (3,0) {};
      \node[label=above right:{$\mathbf{y_2}$}](B) at (2,1) {};
      \node[label=above right:{$\mathbf{y_3}$}](C) at (1,2) {};
      \node[label=above left:{$\mathbf{y_4}$}](D) at (0,3) {};
      \draw[fill] (A) circle (1.5pt);
      \draw[fill] (B) circle (1.5pt);
      \draw[fill] (C) circle (1.5pt);
      \draw[fill] (D) circle (1.5pt);
      \draw[->,thick,cyan] (A) -- (B);
      \draw[->,thick,cyan] (B) -- (C);
      \def\shf{0.2ex};
      \draw[->,transform canvas={xshift=\shf,yshift=\shf},thick, cyan] (C) -- (D);
      \draw[<-,transform canvas={xshift=-\shf,yshift=-\shf},thick, cyan] (C) -- (D);
      \begin{scope}[shift={(\scopesp,0)}]
        \draw[ultra thin,lightgray] (-0.1,-0.1) grid (3.6,3.4);
        \draw[-{Stealth},semithick] (-0.2,0) -- (4,0) node[right] {$X_1$};
        \draw[-{Stealth},semithick] (0,-0.2) -- (0,4) node[above] {$X_2$};
        \node[label=below right:{$\mathbf{y_1}$}](A) at (3,0) {};
        \node[label=above right:{$\mathbf{y_2}$}](B) at (2,1) {};
        \node[label=above right:{$\mathbf{y_3}$}](C) at (1,2) {};
        \node[label=above left:{$\mathbf{y_4}$}](D) at (0,3) {};
        \draw[fill] (A) circle (1.5pt);
        \draw[fill] (B) circle (1.5pt);
        \draw[fill] (C) circle (1.5pt);
        \draw[fill] (D) circle (1.5pt);
        \draw[->,thick,cyan] (A) -- +(-\dirlength, \dirlength);
        \draw[->,thick,purple] (B) -- +(-\dirlength, \dirlength);
        \draw[->,thick,purple] (C) -- +(-\dirlength, \dirlength);
        \draw[->,thick,cyan] (D) -- +(\dirlength, -\dirlength);
      \end{scope}
    \end{tikzpicture}
  }
  \par
  \subfloat[Chamber $\mathcal{C}_2$.]{
    \begin{tikzpicture}[scale=0.9]
      \draw[ultra thin,lightgray] (-0.1,-0.1) grid (3.6,3.4);
      \draw[-{Stealth},semithick] (-0.2,0) -- (4,0) node[right] {$X_1$};
      \draw[-{Stealth},semithick] (0,-0.2) -- (0,4) node[above] {$X_2$};
      \node[label=below right:{$\mathbf{y_1}$}](A) at (3,0) {};
      \node[label=above right:{$\mathbf{y_2}$}](B) at (2,1) {};
      \node[label=above right:{$\mathbf{y_3}$}](C) at (1,2) {};
      \node[label=above left:{$\mathbf{y_4}$}](D) at (0,3) {};
      \draw[fill] (A) circle (1.5pt);
      \draw[fill] (B) circle (1.5pt);
      \draw[fill] (C) circle (1.5pt);
      \draw[fill] (D) circle (1.5pt);
      \draw[->,thick,cyan] (A) -- (B);
      \def\shf{0.2ex};
      \draw[->,transform canvas={xshift=\shf,yshift=\shf},thick, cyan] (B) -- (C);
      \draw[<-,transform canvas={xshift=-\shf,yshift=-\shf},thick, cyan] (B) -- (C);
      \draw[->,thick,cyan] (D) -- (C);
      \begin{scope}[shift={(\scopesp,0)}]
        \draw[ultra thin,lightgray] (-0.1,-0.1) grid (3.6,3.4);
        \draw[-{Stealth},semithick] (-0.2,0) -- (4,0) node[right] {$X_1$};
        \draw[-{Stealth},semithick] (0,-0.2) -- (0,4) node[above] {$X_2$};
        \node[label=below right:{$\mathbf{y_1}$}](A) at (3,0) {};
        \node[label=above right:{$\mathbf{y_2}$}](B) at (2,1) {};
        \node[label=above right:{$\mathbf{y_3}$}](C) at (1,2) {};
        \node[label=above left:{$\mathbf{y_4}$}](D) at (0,3) {};
        \draw[fill] (A) circle (1.5pt);
        \draw[fill] (B) circle (1.5pt);
        \draw[fill] (C) circle (1.5pt);
        \draw[fill] (D) circle (1.5pt);
        \draw[->,thick,cyan] (A) -- +(-\dirlength, \dirlength);
        \draw[->,thick,purple] (B) -- +(-\dirlength, \dirlength);
        \draw[->,thick,purple] (C) -- +(\dirlength, -\dirlength);
        \draw[->,thick,cyan] (D) -- +(\dirlength, -\dirlength);
      \end{scope}
    \end{tikzpicture}
  }
  \par
  \subfloat[Chamber $\mathcal{C}_3$.]{
    \begin{tikzpicture}[scale=0.9]
      \draw[ultra thin,lightgray] (-0.1,-0.1) grid (3.6,3.4);
      \draw[-{Stealth},semithick] (-0.2,0) -- (4,0) node[right] {$X_1$};
      \draw[-{Stealth},semithick] (0,-0.2) -- (0,4) node[above] {$X_2$};
      \node[label=below right:{$\mathbf{y_1}$}](A) at (3,0) {};
      \node[label=above right:{$\mathbf{y_2}$}](B) at (2,1) {};
      \node[label=above right:{$\mathbf{y_3}$}](C) at (1,2) {};
      \node[label=above left:{$\mathbf{y_4}$}](D) at (0,3) {};
      \draw[fill] (A) circle (1.5pt);
      \draw[fill] (B) circle (1.5pt);
      \draw[fill] (C) circle (1.5pt);
      \draw[fill] (D) circle (1.5pt);
      \def\shf{0.2ex};
      \draw[->,transform canvas={xshift=\shf,yshift=\shf},thick, cyan] (A) -- (B);
      \draw[<-,transform canvas={xshift=-\shf,yshift=-\shf},thick, cyan] (A) -- (B);
      \draw[->,thick,cyan] (C) -- (B);
      \draw[->,thick,cyan] (D) -- (C);
      \begin{scope}[shift={(\scopesp,0)}]
        \draw[ultra thin,lightgray] (-0.1,-0.1) grid (3.6,3.4);
        \draw[-{Stealth},semithick] (-0.2,0) -- (4,0) node[right] {$X_1$};
        \draw[-{Stealth},semithick] (0,-0.2) -- (0,4) node[above] {$X_2$};
        \node[label=below right:{$\mathbf{y_1}$}](A) at (3,0) {};
        \node[label=above right:{$\mathbf{y_2}$}](B) at (2,1) {};
        \node[label=above right:{$\mathbf{y_3}$}](C) at (1,2) {};
        \node[label=above left:{$\mathbf{y_4}$}](D) at (0,3) {};
        \draw[fill] (A) circle (1.5pt);
        \draw[fill] (B) circle (1.5pt);
        \draw[fill] (C) circle (1.5pt);
        \draw[fill] (D) circle (1.5pt);
        \draw[->,thick,cyan] (A) -- +(-\dirlength, \dirlength);
        \draw[->,thick,purple] (B) -- +(\dirlength, -\dirlength);
        \draw[->,thick,purple] (C) -- +(\dirlength, -\dirlength);
        \draw[->,thick,cyan] (D) -- +(\dirlength, -\dirlength);
      \end{scope}
    \end{tikzpicture}
  }
  \par
  \subfloat[Chamber $\mathcal{C}_4$.]{
    \begin{tikzpicture}[scale=0.9]
      \draw[ultra thin,lightgray] (-0.1,-0.1) grid (3.6,3.4);
      \draw[-{Stealth},semithick] (-0.2,0) -- (4,0) node[right] {$X_1$};
      \draw[-{Stealth},semithick] (0,-0.2) -- (0,4) node[above] {$X_2$};
      \node[label=below right:{$\mathbf{y_1}$}](A) at (3,0) {};
      \node[label=above right:{$\mathbf{y_2}$}](B) at (2,1) {};
      \node[label=above right:{$\mathbf{y_3}$}](C) at (1,2) {};
      \node[label=above left:{$\mathbf{y_4}$}](D) at (0,3) {};
      \draw[fill] (A) circle (1.5pt);
      \draw[fill] (B) circle (1.5pt);
      \draw[fill] (C) circle (1.5pt);
      \draw[fill] (D) circle (1.5pt);
      \def\shf{0.2ex};
      \draw[->,transform canvas={xshift=\shf,yshift=\shf},thick, cyan] (A) -- (B);
      \draw[<-,transform canvas={xshift=-\shf,yshift=-\shf},thick, cyan] (A) -- (B);
      \def\shf{0.2ex};
      \draw[->,transform canvas={xshift=\shf,yshift=\shf},thick, cyan] (C) -- (D);
      \draw[<-,transform canvas={xshift=-\shf,yshift=-\shf},thick, cyan] (C) -- (D);
      \begin{scope}[shift={(\scopesp,0)}]
        \draw[ultra thin,lightgray] (-0.1,-0.1) grid (3.6,3.4);
        \draw[-{Stealth},semithick] (-0.2,0) -- (4,0) node[right] {$X_1$};
        \draw[-{Stealth},semithick] (0,-0.2) -- (0,4) node[above] {$X_2$};
        \node[label=below right:{$\mathbf{y_1}$}](A) at (3,0) {};
        \node[label=above right:{$\mathbf{y_2}$}](B) at (2,1) {};
        \node[label=above right:{$\mathbf{y_3}$}](C) at (1,2) {};
        \node[label=above left:{$\mathbf{y_4}$}](D) at (0,3) {};
        \draw[fill] (A) circle (1.5pt);
        \draw[fill] (B) circle (1.5pt);
        \draw[fill] (C) circle (1.5pt);
        \draw[fill] (D) circle (1.5pt);
        \draw[->,thick,cyan] (A) -- +(-\dirlength, \dirlength);
        \draw[->,thick,purple] (B) -- +(\dirlength, -\dirlength);
        \draw[->,thick,purple] (C) -- +(-\dirlength, \dirlength);
        \draw[->,thick,cyan] (D) -- +(\dirlength, -\dirlength);
      \end{scope}
    \end{tikzpicture}
  }
  \caption{The four chambers associated to the quadrilateral on a line. Chambers $\mathcal{C}_1$, $\mathcal{C}_2$, $\mathcal{C}_3$ are single-sign-change chambers. Chamber $\mathcal{C}_4$ is {\em not} a single-sign-change chamber because the direction of the vectors \(\mathbf{u}_1,\dots,\mathbf{u}_4\) changes {\em three} times.}
  \label{fig:chamb1234}
\end{figure}

\begin{definition}\label{def:chambers}
  We call respectively the regions in \(\rgz^\vE\) (= \(\rgz^{12}\), by abuse of notation) corresponding to each case of~\cref{th:4gon} the \emph{\(i\)-th chamber}, and we denote them by \(\mathcal{C}_1,\dots,\mathcal{C}_4\).
  We call \emph{single-sign-change chambers} the chambers for which the direction in the sequence of vectors \(\mathbf{u}_1,\dots,\mathbf{u}_4\) changes only one time; in this case these are the chambers \(\mathcal{C}_1,\mathcal{C}_2,\mathcal{C}_3\).
\end{definition}

\begin{theorem}\label{thr:quadrilater-line-main}
  Consider the E-graph \(G=(V,E)\) given by~\cref{fig:4thGhat}.
  Consider a vector of rate constants \(\mathbf{k}\in\rgz^\vE\) and consider \(\mathbf{k}^{*}=(k_1^{*},\dots,k_4^{*})\) defined above.
  The dynamical system generated by \(G\) and \(\mathbf{k}\in\rgz^\vE\) is disguised toric if and only if
  \begin{enumerate}
    \item the vector \(\mathbf{k}\) belongs to the 1st chamber, or to the 2nd chamber, or to the 3rd chamber (i.e. these are the \emph{single-sign-change chambers})

          or
    \item the vector \(\mathbf{k}\) belongs to the 4th chamber and
          \begin{align}\label{eq:Segre}
            k_3^*k_2^*\leq k_4^*k_1^*.
          \end{align}
  \end{enumerate}
\end{theorem}
\begin{proof}
  The first case is a particular case of a more general fact proved in~\cref{thr:cycles-are-done}.
  \medskip

  Fix \(\mathbf{k}\) in the 4th chamber \(\mathcal{C}_4\).
  Then, the system generated by \(G\), \(\mathbf{k}\) is equal to the system generated by the E-graph \(G^{*}\) given by (D) in~\cref{fig:chamb1234} and \(\mathbf{k}^{*}\).
  Hence, we may restrict to the system generated by \(G^{*}\), \(\mathbf{k}^{*}\).

  Now, we consider the E-graph \(\hat{G}\) given by~\cref{fig:4thchamber-detail-completion}, which contains the same source vertices as the E-graph \(G^{*}\); we will obtain the desired result by considering the detailed balance conditions for \(\hat{G}\), as explained below.
  We also consider the rate constants \(\hat{\mathbf{k}}\) given by
  \begin{align*}
    \begin{aligned}
      \hat{k}_{12} & \coloneqq k_1^* \\
      \hat{k}_{21} & \coloneqq k_{2}^*+a
    \end{aligned} & & \begin{aligned}
      \hat{k}_{23} & \coloneqq a \\
      \hat{k}_{32} & \coloneqq b
    \end{aligned} & & \begin{aligned}
      \hat{k}_{43} & \coloneqq k_4^* \\
      \hat{k}_{34} & \coloneqq k_3^*+b
    \end{aligned}
  \end{align*} where \(a,b>0\). First we will show that, for every \(\mathbf{k}^{*}\) satisfying~\cref{eq:Segre}, there exist \(a,b>0\) for which the couple \((\hat{G},\hat{\mathbf{k}})\) satisfies the detailed balance condition (and then also the complex balanced condition).
  Second, we will show that if the system generated by \(G^{*}\), \(\mathbf{k}^{*}\) is disguised toric, then the condition~\cref{eq:Segre} is necessarily satisfied.
  \medskip

  The dynamical system generated by \(G^{*}\) and \(\mathbf{k}^{*}\) is
  \begin{align}\label{eq:diff1}
    \frac{\dd x_1}{\dd t} & = -k_{12}^* x_1^3 +k_{21}^* x_1^2 x_2 -k_{34}^* x_1 x_2^2 +k_{43}^* x_2^3;\\
    \frac{\dd x_2}{\dd t} & =-\frac{\dd x_1}{\dd t}.
  \end{align}
  So, given a positive steady state \((\tilde{x}_1,\tilde{x}_2)\in\rgz^2\), the ratio
  \(\alpha\coloneqq \frac{\tilde{x}_2}{\tilde{x}_1}>0\) satisfies the equation
  \begin{align}
    \label{eq:alpha-and-ks}
    k_1^*-\alpha k_2^*-\alpha^2 (\alpha k_{4}^*-k_{3}^*)=0.
  \end{align}
  From this equation and assumption~\cref{eq:Segre} it follows that
  \[
    \frac{k_{3}^*}{k_{4}^*}\leq \alpha\leq \frac{k_1^{*}}{k_2^{*}}.
  \] Hence, we may set \(b\coloneqq \alpha k_4^{*}-k_3^{*}\) and \(a:= \alpha b\).
  It is trivial to check that the point \((\tilde{x}_1,\tilde{x}_2)\) satisfies the detailed balance conditions for the couple \((\hat{G}, \hat{\mathbf{k}})\) if and only if
  \begin{align*}\label{eq:equiv}
    \frac{\hat{k}_{12}}{\hat{k}_{21}}=
    \frac{\hat{k}_{23}}{\hat{k}_{32}}=
    \frac{\hat{k}_{34}}{\hat{k}_{43}}=\alpha,
  \end{align*} which are satisfied for such values of \(a\) and \(b\).

  \medskip

  Now assume that~\cref{eq:Segre} is \emph{not} satisfied.
  If the system generated by \(G^{*}\) and \(\mathbf{k}^{*}\) is disguised toric, by \cite[Theorem 4.7]{cjy}, we should be able to find a complex balanced realization using the complete directed E-graph \(G\) given by~\cref{fig:4thGhat}.
  We will focus on the vertex \(\by_4\), since we just need to check that the complex balanced condition fails at one vertex.
  Consider the following realization using the E-graph \(G'=G\) and \(\mathbf{k}'\) given by
  \[
    k'_{43}\coloneqq a\quad k'_{42}\coloneqq \frac{1}{2}b\quad k'_{41}\coloneqq \frac{1}{3}(k^{*}_{4}-a-b)\quad k'_{34}\coloneqq k_3^{*}+a+b,
  \] where \(a,b>0\) and \(a+b<k'_{43}\) (we just focus on the relevant rate constants).

  Again, fix a positive steady state \((\tilde{x}_1,\tilde{x}_2)\) of the system generated by \(G^{*}\) and \(\mathbf{k}^{*}\).
  If for this steady state we have \(k_{3}^* \tilde{x}_1 \tilde{x}_2^2<k_{4}^* \tilde{x}_2^3\), then, by~\cref{eq:alpha-and-ks}, \cref{eq:Segre} is satisfied.
  Hence, we assume that \(k_{3}^* \tilde{x}_1 \tilde{x}_2^2>k_{4}^* \tilde{x}_2^3\).
  So, for the outflow at \(\mathbf{y}_4\) we have
  \[
    (k'_{43}+k'_{42}+k'_{41})\tilde{x}_2^3<k_{4}^* \tilde{x}_2^3<k_{3}^* \tilde{x}_1 \tilde{x}_2^2<k'_{34} \tilde{x}_1 \tilde{x}_2^2,
  \] and the inflow is greater than \(k'_{34} \tilde{x}_1 \tilde{x}_2^2\).
  From this, we conclude that the complex balanced condition cannot be satisfied.
\end{proof}

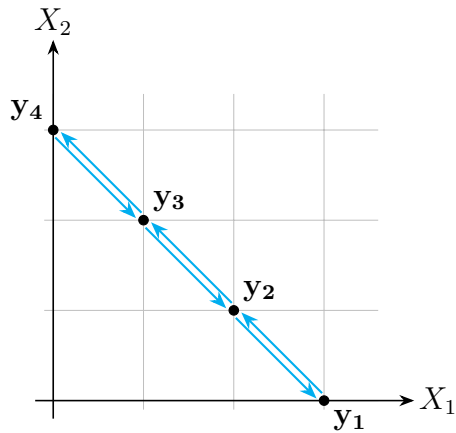
\begin{figure}
  \centering
  \begin{tikzpicture}[scale=1.2]
    \draw[ultra thin,lightgray](-0.1,-0.1) grid (3.6,3.4);
    \draw[-{Stealth},semithick] (-0.2,0) -- (4,0) node[right] {$X_1$};
    \draw[-{Stealth},semithick] (0,-0.2) -- (0,4) node[above] {$X_2$};
    \node[label=below right:{$\mathbf{y_1}$}](A) at (3,0) {};
    \node[label=above right:{$\mathbf{y_2}$}](B) at (2,1) {};
    \node[label=above right:{$\mathbf{y_3}$}](C) at (1,2) {};
    \node[label=above left:{$\mathbf{y_4}$}](D) at (0,3) {};
    \draw[fill] (A) circle (1.5pt);
    \draw[fill] (B) circle (1.5pt);
    \draw[fill] (C) circle (1.5pt);
    \draw[fill] (D) circle (1.5pt);
    \def\shf{0.2ex};
    \draw[->,transform canvas={xshift=\shf,yshift=\shf},thick, cyan] (A) -- (B);
    \draw[<-,transform canvas={xshift=-\shf,yshift=-\shf},thick, cyan] (A) -- (B);
    \draw[->,transform canvas={xshift=\shf,yshift=\shf},thick, cyan] (B) -- (C);
    \draw[<-,transform canvas={xshift=-\shf,yshift=-\shf},thick, cyan] (B) -- (C);
    \draw[->,transform canvas={xshift=\shf,yshift=\shf},thick, cyan] (C) -- (D);
    \draw[<-,transform canvas={xshift=-\shf,yshift=-\shf},thick, cyan] (C) -- (D);
  \end{tikzpicture}
  \caption{Detailed balanced extension of \(\mathcal{C}_4\).}
  \label{fig:4thchamber-detail-completion}
\end{figure}

\section{Globally stable systems which are not disguised toric}
\label{sec:globally-stable-not-disguised-toric}

In this section we discuss the difference between uniqueness of equilibria and the property of being disguised toric.
In particular, we provide examples of dynamical systems which are globally stable but fail to be disguised toric.
We study the qualitative behaviour of the dynamical systems inside a fixed stoichiometric compatibility class, i.e., up to conservation laws. See \cite[Definition 3.4.6.]{feinberg} for a definition.

\begin{lemma}\label{lem:discrim}
  Consider the E-graph \(G^{*}\) given by~(D) of \cref{fig:chamb1234}.
  The dynamical system generated by $G$, $\mathbf{k}^{*}$ has exactly one equilibrium point in each stoichiometric compatibility class if and only if the following inequality is satisfied
  \begin{equation}\label{eq:delta}
    (k^*_{3} k^*_{2})^2-4k^*_{4}(k^*_{2})^3-4(k^*_{3})^3k^*_{1}-27(k^*_{4}k^*_{1})^2+18k^*_{4}k^*_{3}k^*_{1}k^*_{2}<0.
  \end{equation}
\end{lemma}

\begin{proof}
  The dynamical system generated by $G^{*}$ and $\mathbf{k}$ is~\cref{eq:diff1}. Since we are interested in the dynamics inside a fixed stoichiometric compatibility class, we have $x_1+x_2$ is constant (see Figure \ref{fig:equil}), since $\frac{\dd x_2}{\dd t}=-\frac{\dd x_1}{\dd t}$.
  Given a steady state of the system generated by $G$ and $\mathbf{k}^*$, the ratio $\alpha\coloneqq \frac{\tilde{x}_2}{\tilde{x}_1}$ satisfies~\cref{eq:alpha-and-ks}.
  Thus we are interested in the zeros of the cubic polynomial:
  \begin{equation*}
    f(\alpha)\coloneqq k_{1}^* -k_{2}^* \alpha +k_{3}^* \alpha^2 -k_{4}^* \alpha^3.
  \end{equation*}

  The dynamical system (\ref{eq:diff1}) has no negative equilibria, since the cubic polynomial $f$ has no negative real roots.
  This follows from Descartes's rule of signs (\cite[Theorem 2.33]{BPR}) for counting positive roots of a real polynomial in one variable.
  Namely, we have the derivative $f'(\alpha)= -k_{2}^* +2 k_{3}^* \alpha -3 k_{4}^* \alpha^2.$ Since $f'(-\alpha)= -k_{2}^* -2 k_{3}^* \alpha -3 k_{4}^* \alpha^2,$ the number of sign-changes in the coefficients of  $f'(-\alpha)$ is zero, thus $f'$ has no negative real roots.
  In addition, by Descartes' rule of signs we obtain that $f$ can either have one or three positive real roots, counted with multiplicity.

  Since $f$ is a cubic polynomial in $\alpha,$ with positive real coefficients $k_{i}^*,$ condition (\ref{eq:delta}) is equivalent with the discriminant of $f$ being negative, i.e., $f$ has one real root and two complex conjugate roots. See \cite[Subsection 4.1]{BPR}. In other words, if (\ref{eq:delta}) holds, then in each stoichiometric compatibility class there exists a unique positive equilibrium.
\end{proof}

\begin{figure}
  \centering
  \begin{minipage}[b]{0.45\linewidth}
    \includegraphics[scale=0.2]{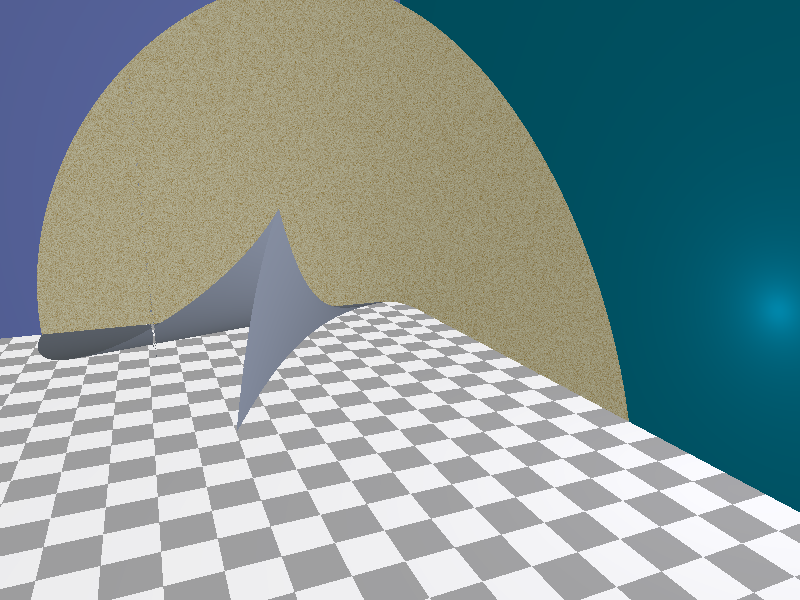}
  \end{minipage}
  \begin{minipage}[b]{0.45\linewidth}
    \includegraphics[scale=0.2]{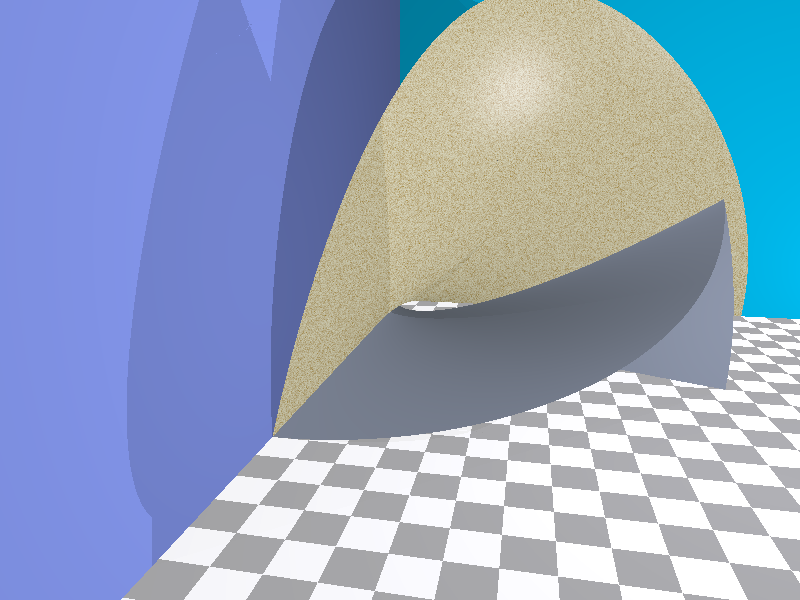}
  \end{minipage}
  \caption{Two different perspectives of the zero locus of the discriminant $\Delta=0$ (the grey hypersurface), in the positive orthant, where the equation (\ref{eq:delta}) is scaled by setting $k_{1}^*=1$. For the positive parameters $\mathbf{k}$ situated above the discriminant surface and in the positive orthant, the dynamical system has exactly one equilibrium point. The dynamical systems corresponding to positive parameters $\mathbf{k}$ below the discriminant surface are multistationary. The beige surface is the Segre variety given by $k_{3}^*k_{2}^*-k_{4}^*k_{1}^*=0$. All the parameters $\mathbf{k}$ on the surface and above it (\ref{eq:Segre}) give rise to complex balance dynamical systems. In particular, above the beige surface the dynamical systems are disguised toric. \label{fig:discriminant}}
\end{figure}

\begin{remark}
  The hypersurface $k_{3}^*k_{2}^*-k_{4}^*k_{1}^*=0$ appears in this context in \cite[page 74]{shiuTh}, where the author studied the toric locus of this reaction network. Here we prove that the disguised toric locus is the open set $k_{3}^*k_{2}^*-k_{4}^*k_{1}^*\leq 0$.
\end{remark}

We are working in the two-dimensional setting and with one-dimensional stoichiometric compatibility class. In addition, a simple computation shows that near the axes the direction of the vector field given by \cref{eq:diff1} points towards the interior of a fixed stoichiometric compatibility class, as in Figure \ref{fig:equil}. Thus if there exists a unique equilibrium, then this equilibrium is also a globally attracting point (see for example Figure \ref{fig:equil}).
\begin{figure}
  \centering
  \begin{tikzpicture}[scale=0.8,>={Latex}]
    \def\dd{1.8pt}
    \draw[-{Stealth},semithick] (-0.2,0) -- (4.6,0) node[right] {$x_1$};
    \draw[-{Stealth},semithick] (0,-0.2) -- (0,4.6) node[above] {$x_2$};
    \node at (2,0) [below]{$\Delta<0$};
    \draw[name path=ss, red] (0,0) -- (2.3,3.6);
    \foreach \l in {2,3,4}{
      \draw[name path=sto\l] (\l,0) -- (0,\l);
      \draw[fill=red, name intersections={of=ss and sto\l, total=\t}]
      \foreach \s in {1,...,\t}{(intersection-\s) circle (\dd)};
      \draw[->, name intersections={of=ss and sto\l, total=\t}]
      \foreach \s in {1,...,\t}{($(intersection-\s)!0.251!(\l,0)$) -- ($(intersection-\s)!0.25!(\l,0)$)};
      \draw[->, name intersections={of=ss and sto\l, total=\t}]
      \foreach \s in {1,...,\t}{($(intersection-\s)!0.651!(\l,0)$) -- ($(intersection-\s)!0.65!(\l,0)$)};
      \draw[->, name intersections={of=ss and sto\l, total=\t}]
      \foreach \s in {1,...,\t}{($(intersection-\s)!0.251!(0,\l)$) -- ($(intersection-\s)!0.25!(0,\l)$)};
      \draw[->, name intersections={of=ss and sto\l, total=\t}]
      \foreach \s in {1,...,\t}{($(intersection-\s)!0.651!(0,\l)$) -- ($(intersection-\s)!0.65!(0,\l)$)};
    };
    \begin{scope}[shift={(7,0)}]
      \draw[-{Stealth},semithick] (-0.2,0) -- (4.6,0) node[right] {$x_1$};
      \draw[-{Stealth},semithick] (0,-0.2) -- (0,4.6) node[above] {$x_2$};
      \node at (2,0) [below]{$\Delta>0$};
      \draw[name path=ss1, red] (0,0) -- (1,3.6);
      \draw[name path=ss2, red] (0,0) -- (2.3,2.8);
      \draw[name path=ss3, red] (0,0) -- (3.6,1);
      \foreach \l in {2,3,4}{
        \draw[name path=sto1\l] (\l,0) -- (0,\l);
        \draw[fill=red, name intersections={of=ss1 and sto1\l, total=\t}]
        \foreach \s in {1,...,\t}{(intersection-\s) coordinate (iss1-\l) circle (\dd)};
        \draw[->] ($(iss1-\l)!0.31!(0,\l)$) -- ($(iss1-\l)!0.3!(0,\l)$);
        \draw[fill=red, name intersections={of=ss3 and sto1\l, total=\t}]
        \foreach \s in {1,...,\t}{(intersection-\s) coordinate (iss3-\l) circle (\dd)};
        \draw[->] ($(iss3-\l)!0.31!(\l,0)$) -- ($(iss3-\l)!0.3!(\l,0)$);
        \draw[fill=red, name intersections={of=ss2 and sto1\l, total=\t}]
        \foreach \s in {1,...,\t}{(intersection-\s) coordinate (iss2-\l) circle (\dd)};
        \draw[->] ($(iss3-\l)!0.451!(iss2-\l)$) -- ($(iss3-\l)!0.45!(iss2-\l)$);
        \draw[->] ($(iss1-\l)!0.31!(iss2-\l)$) -- ($(iss1-\l)!0.3!(iss2-\l)$);
      };
    \end{scope}
  \end{tikzpicture}
  \caption{The two-dimensional phase plane $x_1 O x_2$ for the system (\ref{eq:diff1}): if there exists a unique equilibrium (the red points), then the equilibrium is also a globally attracting point. This holds since each stoichiometric compatibility class (the segments where $x+y=\mathrm{constant}$ in $\rgz^2$) is one-dimensional.}
  \label{fig:equil}
\end{figure}
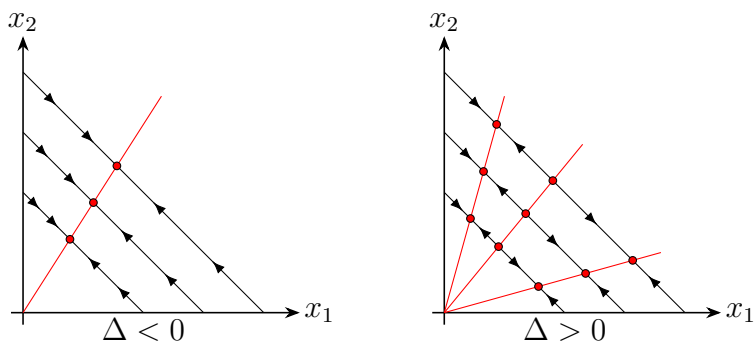

Hence for the points in the parameter space situated above the zero locus of the discriminant in Figure \ref{fig:discriminant} (where we dehomogenise the space of rate constants by setting $k_{1}^* = 1$), there exists a globally attracting fixed point for each stoichiometric compatibility class (i.e., up to conservation law).

In other words, the zero locus of the discriminant (the grey hypersurface), completely separates the globally stable dynamical systems (those corresponding to the parameters above the grey surface) and the multistationary dynamical systems (those corresponding to the parameters below the grey surface). Note that a dynamical system being multistationary or not is an important property of bio-chemical reaction networks (see \cite{BD}), since multistationarity can be translated into distinct responses of the cells, in function of their initial conditions (up to conservation law).

One can show that inequality (\ref{eq:Segre}) implies (\ref{eq:delta}). This has the following interpretation: there are dynamical systems which have a single equilibrium in each stoichiometric compatibility class, but that fail to be disguised toric. This is because at the equilibrium point the complex balance conditions fail to be satisfied.

\section{Filling an empty toric locus}\label{sec:rectange-E-graph}
In this section, we study the E-graph \(G=(V,E)\) given by~\cref{fig:4reactions}, which is a generalization of \cite[Example 5.2]{cjy}.
Here, we consider
\begin{align}\label{eq:complBal00}
  \begin{aligned}
    \mathbf{y}_1\coloneqq
    \begin{pmatrix*}
      0 \\ 0
    \end{pmatrix*} \\
    \mathbf{y}_5\coloneqq \mathbf{y}_1+
    \begin{pmatrix*}
      \alpha A \\ \beta B
    \end{pmatrix*}
  \end{aligned} & &
                    \begin{aligned}
                      \mathbf{y}_2\coloneqq
                      \begin{pmatrix*}
                        A \\ 0
                      \end{pmatrix*} \\
                      \mathbf{y}_6\coloneqq \mathbf{y}_2+
                      \begin{pmatrix*}
                        -\alpha A \\ \beta B
                      \end{pmatrix*}
                    \end{aligned} & &
                                      \begin{aligned}
                                        \mathbf{y}_3\coloneqq
                                        \begin{pmatrix*}
                                          A \\ B
                                        \end{pmatrix*} \\
                                        \mathbf{y}_7\coloneqq \mathbf{y}_3+
                                        \begin{pmatrix*}
                                          -\alpha A \\ -\beta B
                                        \end{pmatrix*}
                                      \end{aligned} & &
                                                        \begin{aligned}
                                                          \mathbf{y}_4\coloneqq
                                                          \begin{pmatrix*}
                                                            0 \\ B
                                                          \end{pmatrix*} \\
                                                          \mathbf{y}_8\coloneqq \mathbf{y}_4+
                                                          \begin{pmatrix*}
                                                            \alpha A \\ -\beta B
                                                          \end{pmatrix*}
                                                        \end{aligned}
\end{align}

with $A,B>0$, $\alpha,\beta\ge 0$ and $\alpha\beta>0$.
When $\alpha,\beta>0$, the E-graph $G$ is not weakly reversible and its toric locus \(K(G)\) is empty.
But, as \cref{coro:dis-toric-locus-rectangle} shows, its disguised toric locus \(\hat{K}(G)\) is a semialgebraic set of Lebesgue positive measure.
Instead, when $\alpha=0$ or $\beta=0$, the E-graph $G$ is weakly reversible and its toric locus $K(G)$ is not empty, it is a hypersurface in $\rgz^E$.
But now, $\hat{K}(G)=K(G)$, hence the disguised toric locus $\hat{K}(G)$ has mesure zero (see~\cref{rmk:rectangele-tangent}).

\begin{figure}[H]
  \centering
  \begin{tikzpicture}[scale=.8]
    \draw[ultra thin,lightgray](-0.1,-0.1) grid (6.6,4.4);
    \draw[-{Stealth},semithick] (-0.2,0) -- (6.8,0) node[right] {$X_1$};
    \draw[-{Stealth},semithick] (0,-0.2) -- (0,4.6) node[above] {$X_2$};
    \fill[white] (2.4,-0.2) rectangle (3.6,0.2);
    \draw[line width=2pt, line cap=round, dash pattern=on 0pt off 3\pgflinewidth] (2.7,0) -- (3.4,0);
    \fill[white] (-0.2,1.4) rectangle (0.2,2.6);
    \draw[line width=2pt, line cap=round, dash pattern=on 0pt off 3\pgflinewidth] (0,1.7) -- (0,2.4);
    \node[label=below left:{$\mathbf{y_1}$}](A) at (0,0) {};
    \node[label=below right:{$\mathbf{y_2}$}](B) at (6,0) {};
    \node[label=above right:{$\mathbf{y_3}$}](C) at (6,4) {};
    \node[label=above left:{$\mathbf{y_4}$}](D) at (0,4) {};
    \def\stx{1.1};
    \def\sty{1.5};
    \path (A) +(\stx,\sty) coordinate (P);
    \path (B) +(-\stx,\sty) coordinate (Q);
    \path (C) +(-\stx,-\sty) coordinate (R);
    \path (D) +(\stx,-\sty) coordinate (S);
    \foreach \l/\t in {A/P,B/Q,C/R,D/S}{
      \draw[fill] (\l) circle (1.5pt);
      \draw[->,thick,orange] (\l) -- (\t);}
    \foreach \l in {P,Q,R,S} {\draw[fill] (\l) circle (1.5pt);}
    \node[label=below:{$\mathbf{y_5}$}] at (P) {};
    \node[label=below:{$\mathbf{y_6}$}] at (Q) {};
    \node[label=above:{$\mathbf{y_7}$}] at (R) {};
    \node[label=above:{$\mathbf{y_8}$}] at (S) {};
  \end{tikzpicture}
  \caption{Four reactions that start at the corners of a rectangle: graph $G$, rates $k_i>0$.}
  \label{fig:4reactions}
\end{figure}
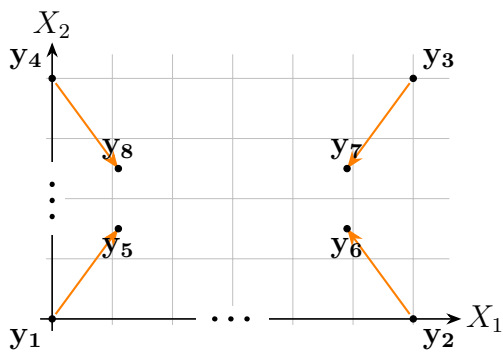

Following \cite[Theorem 4.7]{cjy}, we consider the realization using the complete directed graph on the sources of \(G\), as in Figure~\ref{fig:rectange-complete}.

\begin{figure}[H]
  \centering
  \begin{tikzpicture}[scale=.8, inner sep=6pt]
    \draw[ultra thin,lightgray](-0.1,-0.1) grid (6.6,4.4);
    \draw[-{Stealth},semithick] (-0.2,0) -- (6.8,0) node[right] {$X_1$};
    \draw[-{Stealth},semithick] (0,-0.2) -- (0,4.6) node[above] {$X_2$};
    \fill[white] (2.4,-0.2) rectangle (3.6,0.2);
    \draw[line width=2pt, line cap=round, dash pattern=on 0pt off 3\pgflinewidth] (2.7,0) -- (3.4,0);
    \fill[white] (-0.2,1.4) rectangle (0.2,2.6);
    \draw[line width=2pt, line cap=round, dash pattern=on 0pt off 3\pgflinewidth] (0,1.7) -- (0,2.4);
    \node[label=below left:{$\mathbf{y_1}$}](A) at (0,0) {};
    \node[label=below right:{$\mathbf{y_2}$}](B) at (6,0) {};
    \node[label=above right:{$\mathbf{y_3}$}](C) at (6,4) {};
    \node[label=above left:{$\mathbf{y_4}$}](D) at (0,4) {};
    \def\stx{1.1};
    \def\sty{1.5};
    \path (A) +(\stx,\sty) coordinate (P);
    \path (B) +(-\stx,\sty) coordinate (Q);
    \path (C) +(-\stx,-\sty) coordinate (R);
    \path (D) +(\stx,-\sty) coordinate (S);
    \foreach \l/\t in {A/P,B/Q,C/R,D/S}{
      \draw[fill] (\l) circle (1.5pt);
      \draw[->,thick,orange] (\l) -- (\t);}
    \foreach \l in {P,Q,R,S} {\draw[fill] (\l) circle (1.5pt);}
    \node[label=below:{$\mathbf{y_5}$}] at (P){};
    \node[label=below:{$\mathbf{y_6}$}] at (Q){};
    \node[label=above:{$\mathbf{y_7}$}] at (R){};
    \node[label=above:{$\mathbf{y_8}$}] at (S){};
    \def\shft{0.5ex};
    \foreach \s/\t in {A/B,D/C}{
      \draw[->,transform canvas={yshift=-\shft},thick, cyan] (\s) -- (\t);
      \draw[<-,transform canvas={yshift=\shft},thick,cyan] (\s) -- (\t);
    }
    \foreach \s/\t in {B/C,D/A}{
      \draw[->,transform canvas={xshift=-\shft},thick, cyan] (\s) -- (\t);
      \draw[<-,transform canvas={xshift=\shft},thick,cyan] (\s) -- (\t);
    }
    \def\shf{0.2ex};
    \draw[->,transform canvas={xshift=\shf,yshift=-\shf},thick, cyan] (A) -- (C);
    \draw[->,transform canvas={xshift=-\shf,yshift=\shf},thick, cyan] (C) -- (A);
    \draw[->,transform canvas={xshift=\shf,yshift=\shf},thick, cyan] (D) -- (B);
    \draw[->,transform canvas={xshift=-\shf,yshift=-\shf},thick, cyan] (B) -- (D);
  \end{tikzpicture}
  \caption{Complete graph over the sources of $G$: graph $\hat{G}$, rates $\hat{k}_i\ge 0$.}
  \label{fig:rectange-complete}
\end{figure}
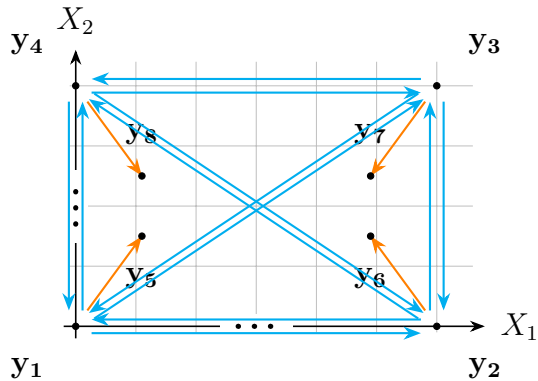

\begin{theorem}\label{th:diagrectangle}
  The dynamical system generated by the E-graph $G=(V,E)$ given by Figure \ref{fig:4reactions} and \(\mathbf{k}\in\rgz^E\) has a complex balanced realization using the E-graph $\hat{G}=(\hat{V},\hat{E})$ given by~\cref{fig:rectange-complete} if and only if
  \begin{equation}\label{eq:rectangle-disg-locus-open}
    \Bigl(\frac{\alpha-\beta}{\alpha+\beta}\Bigr)^2<\frac{k_1 k_3}{k_2 k_4}<\Bigl(\frac{\alpha+\beta}{\alpha-\beta}\Bigr)^2.
  \end{equation}
\end{theorem}
\begin{proof}

  Denote respectively by $k_i$ and $\hat{k}_{ij}>0$ the rate of the reaction $\mathbf{y_i}\rightarrow \mathbf{y_j}$ in $G$ and $\hat{G}$.
  Fix rate constants \(\mathbf{k}\in\rgz^E\) and consider rate constants \(\hat{\mathbf{k}}\in\rgz^{\hat{E}}\) given by
  \begin{align}\label{eq:embed01}
    \begin{aligned}
      \hat{k}_{12}&\coloneqq k_1(\alpha-a)\\
      \hat{k}_{13}&\coloneqq k_1 a\\
      \hat{k}_{14}&\coloneqq k_1(\beta-a)
    \end{aligned} & &
                      \begin{aligned}
                        \hat{k}_{21}&\coloneqq k_2(\alpha-b)\\
                        \hat{k}_{23}&\coloneqq k_2(\beta-b)\\
                        \hat{k}_{24}&\coloneqq k_2 b
                      \end{aligned} & &
                                        \begin{aligned}
                                          \hat{k}_{31}&\coloneqq k_3 c\\
                                          \hat{k}_{32}&\coloneqq k_3(\beta-c)\\
                                          \hat{k}_{34}&\coloneqq k_3(\alpha-c)
                                        \end{aligned} & &
                                                          \begin{aligned}
                                                            \hat{k}_{41}&\coloneqq k_4(\beta-d)\\
                                                            \hat{k}_{42}&\coloneqq k_4 d\\
                                                            \hat{k}_{43}&\coloneqq k_4(\alpha-d).
                                                          \end{aligned}
  \end{align} When
  \begin{align*}
    0<a,b,c,d<\min\{\alpha,\beta\},
  \end{align*} we have \(\hat{k}_{ij}>0\) and the systems generated by \(G,\mathbf{k}\) and by \(\hat{G}, \hat{\mathbf{k}}\) are equal.

  In \cite[Theorem 9]{MR2561288}, the authors present a method to derive the equations defining the toric locus of a graph from its Laplacian and the nullspace of its Cayley matrix.
  Following this method, we computed the equations for the toric locus of $\hat{G}$.
  It is given by a single equation
  \begin{equation*}
    p(\hat{\mathbf{k}})=0,
  \end{equation*} where $p(\hat{\mathbf{k}})$ is an homogeneous polynomial of degree 6 and 346 terms, we do not reproduce it here.
  Observe that both the Laplacian and the nullspace of the Cayley matrix of \(\hat{G}\) do \emph{not} depend on $A,B$.
  Hence, the polynomial $p(\hat{\mathbf{k}})$ does not depend on \(A,B\).
  \medskip

  Evaluating the polynomial $p(\hat{\mathbf{k}})$ at the $\hat{\mathbf{k}}$ given by~\cref{eq:embed01}, we obtain the equation
  \begin{equation}\label{eq:disguised-toric-rectangle}
    k_1k_2k_3k_4\bigl(k_1k_3(a+c-(\alpha+\beta))^2-k_2k_4(b+d-(\alpha+\beta))^2\bigr)=0.
  \end{equation} Since \(k_1,\dots,k_4>0\),~\cref{eq:disguised-toric-rectangle} is equivalent to
  \begin{equation}\label{eq:disguised-toric-rectangle2}
    \frac{k_1k_3}{k_2k_4} = \biggl(\frac{b+d-(\alpha+\beta)}{a+c-(\alpha+\beta)}\biggr)^2.
  \end{equation}

  The system generated by \(G,\mathbf{k}\) has a toric realisation using the graph \(\hat{G}\) if and only if there exist \(0<a,b,c,d<\min\{\alpha,\beta\}\) satisfying~\cref{eq:disguised-toric-rectangle} or, equivalently,~\cref{eq:disguised-toric-rectangle2}.
  Hence, in order to determine when such a realisation exists, we need to find the relative maximum and minimum of the function
  \[
    \biggl(\frac{b+d-(\alpha+\beta)}{a+c-(\alpha+\beta)}\biggr)^2
  \] restricted to \(0<a,b,c,d<\min\{\alpha,\beta\}\).
  First assume that $\alpha\not=\beta$.
  The function \(f(x)=(x-(\alpha+\beta))^2\) is a parabola with a double zero at \(x=\alpha+\beta\), which is bigger than \(2\min\{\alpha,\beta\}\).
  Hence, in the region \(0<x<2\min\{\alpha,\beta\}\), its maximum is \((\alpha+\beta)^2\) at \(x=0\) and its minimum is \((\alpha-\beta)^2\) at \(x=2\min\{\alpha,\beta\}\).
  Hence, the system generated by \(G,\mathbf{k}\) has a toric realisation using the graph \(\hat{G}\) if
  \begin{equation}\label{eq:alphaBeta}
    \Bigl(\frac{\alpha-\beta}{\alpha+\beta}\Bigr)^2<\frac{k_1 k_3}{k_2 k_4}<\Bigl(\frac{\alpha+\beta}{\alpha-\beta}\Bigr)^2.
  \end{equation}
  For the case $\alpha=\beta$, the reaction network $G$ can always be realised using a single target network, where the single target is the intersection of the diagonals of the rectangle given by $\mathbf{y}_1,\dots,\mathbf{y}_4$.
  Hence, by \cite{cjy-singleT}~for all $k_1,\dots,k_4$ the system generated by $G$ and $\mathbf{k}$ is disguised toric.
  Observe that, fixing $\beta>0$ and considering the limit $\alpha\to\beta$ of~\cref{eq:rectangle-disg-locus-open}, the~\cref{eq:disguised-toric-rectangle} imposes no restriction on $k_1,\dots,k_4$, which agrees with the previous fact.
\end{proof}

\begin{corollary}\label{coro:dis-toric-locus-rectangle}
  We have $\hat{K}(G)$ is the set of $k_1,\dots,k_4>0$ such that
  \[
    \Bigl(\frac{\alpha-\beta}{\alpha+\beta}\Bigr)^2\le\frac{k_1 k_3}{k_2 k_4}\le\Bigl(\frac{\alpha+\beta}{\alpha-\beta}\Bigr)^2.
  \]
\end{corollary}
\begin{proof}
  By \cite[Theorem 4.7]{cjy}, it is enough to look at the complete directed graph on the source vertices of $G$, i.e., $\hat{G}$, while allowing some $\hat{k}_{ij} = 0$ (which practicaly means that we are focusing on $\hat{G}$ and its weakly reversible subgraphs).
  The result follows from Theorem \ref{th:diagrectangle}, together with the remark that if one of the inequalities
  \[
    \Bigl(\frac{\alpha-\beta}{\alpha+\beta}\Bigr)^2\le\frac{k_1 k_3}{k_2 k_4}\le\Bigl(\frac{\alpha+\beta}{\alpha-\beta}\Bigr)^2
  \] is {\em not}~strict, then we can just choose some $\hat{k}_{ij} = 0$ in (\ref{eq:embed01}).
\end{proof}

\begin{remark}\label{rmk:rectangele-tangent}
  Observe that
  $$ \left(\frac{\alpha-\beta}{\alpha+\beta}\right)^2$$ is the square of the tangent of the angle between the vectors $(\alpha,\beta)$ and $(1,1)$, which is a measure of how far is $(\alpha A,\beta B)$ of lying on one of the diagonal of the rectangle given by $\mathbf{y}_1,\dots,\mathbf{y}_4$.
  So, the worst case is when such an angle is $\pi/4$, that is, when $\alpha$ or $\beta$ is zero.
  In this case, the network $G$ is weakly reversible and its toric locus is $k_1k_3=k_2k_4$, which is also its disguised toric locus since $\tan(\pi/4)=1$.
\end{remark}

\section{The $N$-gon on a line}
\label{sec:the-n-gon-in-a-line}
In this section we study the complete E-graph \(G=(V,E)\) over the finite set \(V\subseteq \bR^2\) of nonnegative integer points on the line \(\{X_1+X_2 =N-1\}\) see~\cref{fig:ngon}.
It corresponds to $N$ vertices on a line.
Following the construction and notation introduced in~\cref{sec:quadrilat}, given a vector of rate constants \(\mathbf{k}\in\rgz^E\) we realize the dynamical system generated by \(G\) and \(\mathbf{k}\) by an E-graph with one reaction per source.
Now, there are \(N\) vectors \(\mathbf{u}_i\).
The direction of the vectors \(\mathbf{u}_1\) and \(\mathbf{u}_N\) does not depend on the values of \(\mathbf{k}\), but for every vector \(\mathbf{u}_2,\dots,\mathbf{u}_{N-1}\) there are two possibilities.
Hence, now we need to consider \(2^{N-2}\) chambers in \(\rgz^E\), one for every possible sequence of directions of the vectors \(\mathbf{u}_2,\dots,\mathbf{u}_{N-1}\), and, if in such a sequence there is a unique direction change, we call that chamber a \emph{single-sign-change} chamber.
So, there are \(N-2\) single-sign-change chambers, one for each vector \(\mathbf{u}_2,\dots,\mathbf{u}_{N-1}\).
\medskip

In this section,~\cref{thr:cycles-are-done} below shows, by means of algebraic methods, that for every \(\mathbf{k}\) belonging to a single-sign-change chamber the system generated by \(G\) and \(\mathbf{k}\) is disguised toric.
So, the disguised toric locus \(\hat{K}(G)\) contains at least \(N-2\) regions in \(\rgz^E\) of positive measure,  while the toric locus \(K(G)\) has codimension \(N-2\), and therefore has Lebesgue measure zero (the codimension is given by the deficiency \(\delta=N-2\) of \(G\)).

\begin{remark}
  After a change of coordinates, given by the maximal minors of the negative of the Laplacian of \(G\) (see \cite{MR2561288}), the toric locus $K(G)$ is parametrized by the monomial map $\nu:\bP^1\rightarrow\bP^{N-1}$.
  That is, the toric locus is the rational normal curve in $\bP^{N-1}$ (see \cite[Proposition 5.2.1]{shiuTh}).
\end{remark}

\begin{figure}
  \begin{center}
    \begin{tikzpicture}[scale=2]
      \draw[step=0.3, ultra thin, lightgray](-0.1,-0.1) grid (2.9,2.5);
      \draw[-{Stealth},semithick] (-0.2,0) -- (3.1,0) node[right] {$X_1$};
      \draw[-{Stealth},semithick] (0,-0.2) -- (0,2.6) node[above] {$X_2$};
      \foreach \x in {1,2,3}
      \filldraw[fill]({2.1-0.3*(\x-1)},{0.3*(\x-1)}) coordinate (A\x) node[above right]
      {$\mathbf{y_{\x}}$} circle (0.05);
      \foreach \x in {5,6,7}
      \filldraw[fill] ({2.1-0.3*\x},{0.3*\x}) coordinate (A\x) circle (0.05);
      \node[above right] at (A7) {$\mathbf{y_{N}}$};
      \node[above right] at (A6) {$\mathbf{y_{N-1}}$};
      \node[above right] at (A5) {$\mathbf{y_{N-2}}$};
      \def\stp{0.1};
      \def\stpp{0.3};
      \path (A5) +(\stp,-\stp) coordinate (A51) +(\stpp,-\stpp) coordinate (A52);
      \path (A3) +(-\stp,\stp) coordinate (A31) +(-\stpp,\stpp) coordinate (A32);
      \draw[line width=2pt, line cap=round, dash pattern=on 0pt off 3\pgflinewidth] (A51) -- (A52);
      \draw[line width=2pt, line cap=round, dash pattern=on 0pt off 3\pgflinewidth] (A31) -- (A32);
    \end{tikzpicture}
  \end{center}
  \caption{The $N$-gon on a line.}
  \label{fig:ngon}
\end{figure}
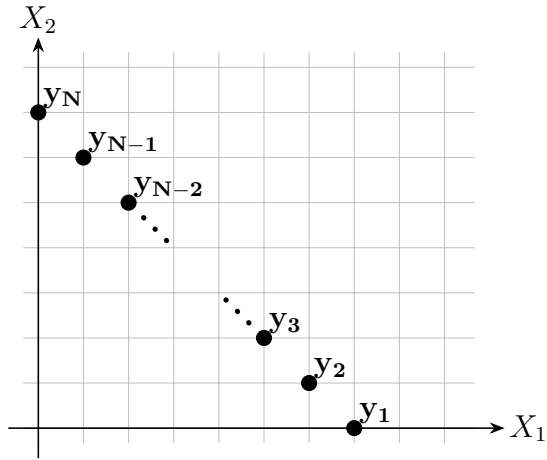

\begin{theorem}\label{thr:cycles-are-done}
  Consider the ``N-gon on a line'' network given by the  \(G=(V,E)\) introduced above.
  Given \(\mathbf{k}\in\rgz^E\) belonging to a single-sign-change chamber, the system generated by \(G\) and \(\mathbf{k}\) is disguised toric.
\end{theorem}
\begin{proof}
  First, we assume that the direction change occurs at the vector \(\mathbf{u}_{N-1}\).
  For example, for \(N=4\) that corresponds to (A) Chamber \(\mathcal{C}_1\) of~\cref{fig:chamb1234}.

  That is, first we show that every system generated by the E-graph \(G^{*}=(V,E^{*})\), where
  \[
    E^{*}\coloneqq \{\by_i\to\by_{i+1}\}_{i=1,\dots,N-1}\bigcup \{\by_N\to\by_{N-1}\},
  \] is disguised toric; then the general case will follow straightforwardly.
  \medskip

  Consider the E-graph \(\hat{G}=(V,\hat{E})\) where
  \[
    \hat{E}\coloneqq \{\by_i\to\by_{i+1},\by_{i+1}\to\by_i\}_{i=1,\dots,N-1}.
  \] The E-graph \(\hat{G}\) is a minimal extension of \(G^{*}\) where the detailed balanced conditions can be established.
  In fact, we will realize the system generated by \(G^{*}\) and \(\mathbf{k}^{*}\) using the graph \(\hat{G}\) in such a way that the detailed balanced condition for \(\hat{G}\) will be satisfied regardless of the values of \(\mathbf{k}^{*}\).

  Given \(\mathbf{k}^{*}\in\rgz^{E^{*}}\) and \(\hat{\mathbf{k}}\in\rgz^{\hat{E}}\) we simplify the notation \(k^{*}_{\by_i\to\by_{i\pm 1}}\) to \(k^{*}_i\) and \(\hat{k}_{\by_i\to\by_j}\) to \(\hat{k}_{ij}\).

  Fix a steady state \((\tilde{x}_0,\tilde{x}_1)\) of the system generated by \(G^{*}\) and \(\mathbf{k}^{*}\).
  The ratio \(\alpha\coloneqq \frac{\tilde{x}_0}{\tilde{x}_1}\) is a root of the polynomial
  \[
    p(x) \coloneqq k^{*}_1x^{N-1}+\dots+k^{*}_{N-1} x-k^{*}_N.
  \]
  Consider \(\hat{\mathbf{k}}\in\rgz^{\hat{E}}\) given by
  \begin{align*}
    \begin{aligned}
      \hat{k}_{12}& \coloneqq k^{*}_1\\
      \hat{k}_{N\,N-1}& \coloneqq k^{*}_N
    \end{aligned} & &
                      \begin{aligned}
                        \hat{k}_{i\,i+1} & \coloneqq k^{*}_i+\hat{k}_{i\,i-1}\\
                        \hat{k}_{i\,i-1} & \coloneqq \alpha\hat{k}_{i-1\, i}\\
                      \end{aligned}
  \end{align*} for every \(i=2,\dots,N-1\).
  Now, the system generated by \(\hat{G}\) and \(\hat{\mathbf{k}}\) is equal to the one generated by \(G^{*}\) and \(\mathbf{k}^{*}\).
  Indeed, for \(i=2,\dots,N-1\), the coefficient of the monomial \(x_0^{N-i}x_1^{i-1}\) is
  \[
    \hat{k}_{i\,i+1}\begin{pmatrix*}[r]
      -1 \\
      1
    \end{pmatrix*} +
    \hat{k}_{i\,i-1}\begin{pmatrix*}[r]
      1 \\
      -1
    \end{pmatrix*} =
    \begin{pmatrix*}[r]
      -k^{*}_i\\
      k^{*}_i
    \end{pmatrix*}.
  \] Moreover, the point \((\tilde{x}_0,\tilde{x}_1)\) is a detailed balance steady state for the system generated by \(\hat{G}\), \(\hat{\mathbf{k}}\) if and only if for every \(i=1,\dots,N-1\)
  \[
    \hat{k}_{i+1\, i} = \alpha\hat{k}_{i\, i+1}.
  \] Hence, by the recursive definition of \(\hat{k}_{i+1\, i}\), we just need to check the case \(i=N-1\).
  Consider the sequence
  \[
    q_{i}\coloneqq \hat{k}_{N-i-1\, N-i},
  \] for \(i=N-1,\dots,1\).
  So, we have \(q_{N-2}=k^{*}_1\) and for \(i=N-3,\dots,0\)
  \[
    q_i=k^{*}_i+q_{i+1}\alpha.
  \]  Hence, by Horner's Method,
  \(\alpha q_0=p(\alpha)+k^{*}_N\) and finally
  \[
    \alpha\hat{k}_{N-1\, N}=\alpha q_0 = k^{*}_N=\hat{k}_{N\,N-1}.\qedhere
  \]
\end{proof}

\begin{remark}
  Another way one can prove this is by using the theory of single target networks. Note that the other chambers will have both disguised-toric and not-disguised-toric points because the presence of three or more sign changes allows us to get two or more equilibria, which rules out disguised-toric.
\end{remark}

\section{An algorithm for computing the disguised toric locus}\label{sect:algor2}
We are now ready to give an algorithm to obtain the disguised toric locus of a reaction network.

First, we introduce some  notation.
By a cone, we mean a polyhedral cone.
Consider two reaction networks $G=(V,E)$, $\hat{G}=(\hat{V}, \hat{E})$.
We consider the locus of $\mathbf{k}\in \rgz^E$ for which the dynamical system generated by $(G,\mathbf{k})$ has a complex balanced realization using the graph $\hat{G}$,
$$\hat{K}(G,\hat{G}):=\{\mathbf{k}\in \rgz^E\ :\ \exists\,\hat{\mathbf{k}}\in\rgz^{\hat{E}} \mbox{ with $F_{G,\mathbf{k}}(\mathbf{x})=F_{\hat{G},\hat{\mathbf{k}}}(\mathbf{x})$ and $(\hat{G},\hat{\mathbf{k}})$ complex balanced}\}.$$

The set $\hat{K}(G,\hat{G})$ is the part of the disguised toric locus of $G$ that can be obtained using the graph $\hat{G}$.
For example, for any network $G$, we have $\hat{K}(G,G)=\mathcal{V}(G,K(G))$.
Another example, for the network $G$ in \cref{sec:complete3pts} we showed that, for all $\mathbf{k}\in \rgz^E$, $\mathcal{V}(G,\mathbf{k})\cap K(G)\not=\emptyset$ which is equivalent to $\hat{K}(G)=\hat{K}(G,G)=\rgz^E$.

Given a $\mathbf{y}\in\R^n$, we denote by $C_{G,\mathbf{y}}$ the positive cone generated by all the reaction vectors $\mathbf{y}\to \mathbf{y}'$ of $G$ (notice that, when $\mathbf{y}$ is not a source of $G$, then $C_{G,\mathbf{y}}=\{0\}$).
We also consider the cone $\Pi_{G, \hat{G}}(\mathbf{y})\subseteq \rgz^{\hat{E}}$ given by the condition on $\hat{\mathbf{k}}\in \rgz^{\hat E}$
\begin{equation*}
  \sum_{\mathbf{y} \to \mathbf{y'} \in \hat G} \hat{{k}}_{\mathbf{y} \to \mathbf{y'}} (\mathbf{y'} - \mathbf{y}) \in  C_{G,\mathbf{y}},
\end{equation*}
and the cone
$$
\Pi_{G,\hat{G}}:=\bigcap_{\mathbf{y} \mbox{\scriptsize{ source of $G$ or }} \hat{G}} \Pi_{G,\hat{G}}(\mathbf{y})\subseteq \rgz^{\hat{E}}.
$$
Observe that $\sum_{\mathbf{y} \to \mathbf{y'} \in \hat G} \hat{{k}}_{\mathbf{y} \to \mathbf{y'}} (\mathbf{y'} - \mathbf{y})$ is the coefficient of $\mathbf{x}^{\mathbf{y}}$ in $F_{\hat{G},\hat{\mathbf{k}}}$.
So, the set \(\Pi_{G,\hat{G}}\) is the locus (possibly empty) of $\hat{\mathbf{k}}\in\rgz^{\hat{E}}$ for which the dynamical system generated by $(\hat{G},\hat{\mathbf{k}})$ has a realization using the network $G$.

For example, vertex $\mathbf{y}_1$ in \cref{fig:rectange-complete} is a source in both the blue ($G$) and the yellow ($\hat{G}$) networks.
For the yellow network, the coefficient of $\mathbf{x}^{\mathbf{y}_1}$ is a multiple of $\mathbf{y}_5-\mathbf{y}_1$, but for the blue network it is a linear combination of three vectors.
Now, such a linear combination is a positive multiple of $\mathbf{y}_5-\mathbf{y}_1$ if and only if $\hat{\mathbf{k}}\in \Pi_{G,\hat{G}}(\mathbf{y}_1)$.

We are interested to parametrize $\Pi_{G,\hat{G}}$, where by to parametrize we mean to give a surjective rational map.
Observe that, since $\Pi_{G,\hat{G}}$ is an intersection of cones, it always admits a parametrization.
Typically, we are interested in cases where $G$ and $\hat{G}$ have the same sources and $C_{G,\mathbf{y}}\subseteq C_{\hat{G},\mathbf{y}}$ for all source.
So, $\dim \rgz^E\le \dim \Pi_{G,\hat{G}}$ and we may use as source for the parametrization some subspace $X\subseteq \rgz^E\times \R^m$, for some $m\ge \dim \Pi_{G,\hat{G}}- \dim \rgz^E$ and with the projection $p\,:\,X\to \rgz^E$ surjective.
For example, consider again \cref{fig:rectange-complete} with the blue ($G$) and the yellow ($\hat{G}$) networks.
Equations~\eqref{eq:embed01} define a parametrization \(\rho: X\to \Pi(G,\hat{G})\) where $X=\rgz^E\times \mathcal{Q}$ with $\mathcal{Q}$ the open cube in $\R^4$ of side length $\min\{\alpha,\beta\}$.

Obviously, among all the parametrizations $\rho:X\to \Pi_{G,\hat G}$ with $X\subseteq \rgz^E\times \R^m$ and $p\,:\,X\to \rgz^E$ surjective, we are interested in those respecting the construction of $F_{G,\mathbf{k}}$ and $F_{\hat{G},\hat{\mathbf{k}}}$. To this intent, we introduce the following restrictions.

Fix a rational map \(\rho:X\to\Pi_{G,\hat G}\) with $X\subseteq \rgz^E\times \R^m$ and with the projection $X\to \rgz^E$ being surjective.
We call \(\rho\) \emph{mass-action faithful} if, for all $(\mathbf{k},\alpha)\in X$, setting $\hat{\mathbf{k}} := \rho(\mathbf{k}, \alpha)$ the couples $(G, {\mathbf{k}})$ and $(\hat G, \hat{\mathbf{k}})$ generate the same dynamical system.

When \(\rho\) is mass-action faithful, we call it \emph{dynamically complete} if, for all $\hat{\mathbf{k}}\in\rgz^{\hat{E}}$, there is $(\mathbf{k},\alpha)\in X$ such that $\rho(\mathbf{k},\alpha)\in \mathcal{V}(\hat{G}, \hat{\mathbf{k}})$.
In other words, if for each dynamical system generated by $\hat{G}$, there is $(\mathbf{k}, \alpha)\in X$ such that $(G,\mathbf{k})$ also generates such a dynamical system.
Notice that if \(\rho\) is surjective, then it is dynamically complete.

We call \(\rho\) a \emph{mass-action parametrization} if it is both mass-action faithful and dynamically complete.

We call \(\rho\) \emph{target-surjective} if, given $\mathbf{k}\in\rgz^E$ and $\hat{\mathbf{k}}\in\rgz^{\hat{E}}$ generating the same dynamical systems, there exists \(\alpha\in\R^m\) such that $(\mathbf{k},\alpha)\in X$ and \(\rho(\mathbf{k},\alpha)=\hat{\mathbf{k}}\).
\medskip

Now, we can state Theorem~\ref{thr:algorithm-based-theorem} on which our algorithm is based.
Its proof follows straightforwardly from the definitions.

\begin{theorem}\label{thr:algorithm-based-theorem}
  Let $G=(V,E)$ and $\hat{G}=(\hat{V},\hat{E})$ be two E-graphs.
  Consider a rational map $\rho:X\to \Pi_{G,\hat{G}}$ with $X\subseteq \rgz^E\times \R^m$ for some $m$ and consider the projection $p:X\to \rgz^E$.
  If \(\rho\) is a mass-action parametrization, then
\[
  \hat{K}(G,\hat{G}) = \mathcal{V}\Bigl(G,p\bigl(\rho^{-1}\bigl(\mathcal{V}(\hat{G},K(\hat{G}))\bigr)\bigr)\Bigr) = \mathcal{V}(G,p(\rho^{-1}(\hat{K}(\hat{G},\hat{G})))).
\] Moreover, if \(\rho\) is also surjective, then
\[
  \hat{K}(G,\hat{G}) = \mathcal{V}(G,p(\rho^{-1}(K(\hat{G})))).
\] Finally, if \(\rho\) is also both surjective and target-surjective, then
\[
  \hat{K}(G,\hat{G}) = p(\rho^{-1}(K(\hat{G}))).
\]
\end{theorem}

\begin{alg}\label{algorithmDisguised1}
  \hfill \\
  \rm{Input:} two reaction networks $G$ and $\hat G$.\\
  \rm{Output:} $$\hat K(G,\hat G)$$

    \textbf{Step 1.} Find a surjective target-surjective mass-action parametrization $$\rho\,:\,  X\to \Pi_{G,\hat G}.$$
    (For example, in Section~\ref{sec:rectange-E-graph}, the map $\rho$ is defined by the  equations~(\ref{eq:embed01}).) 

  \textbf{Step 2.} Compute the equation on $\mathbf{\hat{k}}$ for the toric locus $K(\hat G)$ of $\hat G$, either by eliminating the variables $x_1,\dots,x_n$ from the complex balance conditions (\ref{eq:complexBal}) or by means of the Matrix Tree Theorem (see \cite{MR2561288}).

  \textbf{Step 3.} Compute $\rho^{-1}(K(\hat G))$, which simply corresponds to substituting the components of $\rho$ within the equations for $K(\hat{G})\subseteq \rgz^{\hat{E}}$.
  (See equation~(\ref{eq:subst}) for an example.)

  \textbf{Step 4.} Project $\rho^{-1}(\hat{K}(G))\subseteq X$ to $\R^E$.
  This amounts to eliminating the parameters of $\R^m$ by means of quantifier elimination.
  (For example, in Section~\ref{sec:rectange-E-graph}, equation~(\ref{eq:alphaBeta}) is the result of eliminating $a,b,c,d$ from equation~(\ref{eq:disguised-toric-rectangle}).)
\end{alg}

\begin{remark}
  Observe that, a surjective target-surjective mass-action parametrization \(\rho\) (from Step 1) can be quite intricate to get, or it can lead to a computationally unfeasible quantifier elimination in Step 4.
  As \cref{thr:algorithm-based-theorem} highlights, the whole set $\hat{K}(G,\hat{G})$ can still be obtained with simpler \(\rho\) at expenses of computing some dynamical completions (see Definition~\ref{def:same-dynamic-locus}).
\end{remark}

Consider two reaction networks $G=(V,E)$, $\hat{G}=(\hat{V}, \hat{E})$.
Observe that from \cite[Theorem 4.10]{cjy} follows that
$$
\hat{K}(G) = \bigcup_{G' \mbox{{\scriptsize{~weakly reversible subgraph of }}}  G_{\rm comp}} \hat{K}(G,G'),
$$ where $G_{\rm comp}$ is the complete graph over the sources of $G$.
So, in order to compute the whole disguised toric locus of $G$, we can apply the previous algorithm to all the subgraphs of $G_{\rm comp}$.

Moreover, notice that, when $G'$ is not weakly reversible, $\hat{K}(G,G')=\emptyset$.
Hence,
$$
\hat{K}(G) = \bigcup_{G' \mbox{{\scriptsize{~subgraph of }}}  G_{\rm comp}} \hat{K}(G,G').
$$ So, in order to compute the whole disguised toric locus of $G$, we can apply the previous algorithm to all the subgraphs of $G_{\rm comp}$.
But now, we will show that the previous algorithm can be adapted to compute
\[
  \bigcup_{G' \mbox{\scriptsize{ subgraph of }} \hat{G}}\hat{K}(G,\hat{G})
\] with almost no additional  computational cost.
The main idea follows from the following observation: allowing some of the coordinates of $\mathbf{\hat{k}}$ to be zero is equivalent to considering the subgraph $G'$ corresponding to removing from the graph $\hat{G}$ the reactions corresponding to those vanishing  coordinates.
So, we simply need to replace $\rgz^{\hat{E}}$ by $\rgez^{\hat{E}}$, and take care that all the limit cases behave as expected.

First, it is clear that when some of the coordinates of $\mathbf{\hat{k}}$ are zero, the
complex balance conditions of the new subgraph $G'$ are the complex balance conditions of $\hat G$ evaluating those coordinates of $\hat{\mathbf{k}}$ to zero.
Moreover, it is not hard to see that the equations of $K(G')\subseteq \rgz^{E'}$ are the equations of $K(\hat{G})\subseteq \rgz^{\hat{E}}$ evaluating those coordinates of $\hat{\mathbf{k}}$ to zero.

Now, we extend \(\Pi_{G,\hat{G}}\) to contain the values of $\mathbf{\hat{k}}$ with some zero coordinates:
\[
  \hat \Pi_{G,\hat{G}} = \bigcap_{\mathbf{y} \mbox{\scriptsize{ source of $G$ or }} \hat{G}} C_{G,\mathbf{y}}\cap \overline{C_{\hat{G},\mathbf{y}}}.
\]
In order to find a parametrization \(\rho:X\to\hat \Pi_{G,\hat{G}}\), notice that $\hat \Pi_{G,\hat{G}}$ may be closed or locally closed depending on the dimension and the relative positions of $C_{G,\mathbf{y}}$ and $C_{\hat G,\mathbf{y}}$; in each case so will be $X$.

\begin{alg}[Computing the disguised toric locus]\label{algorithmDisguised2}
  \hfill \\
  \rm{Input:} Two reaction networks $G$ and $\hat G$.\\
  \rm{Output:} $$\bigcup_{G' \mbox{\scriptsize{ subgraph of }}\hat G}\hat K(G,G').$$
  \hfill \\

  \textbf{Step 1.} Find a surjective and target-surjective mass-action parametrization $$\rho\,:\,  X\to \hat\Pi_{G,\hat G}.$$

  \textbf{Step 2.} Reproduce Algorithm~\ref{algorithmDisguised1} using this $\rho$.
\end{alg}

\begin{remark}
  Note that going from Algorithm~\ref{algorithmDisguised1} to \ref{algorithmDisguised2} simply amounts to interchange the strict inequalities on the parameters of $\R^m$ defining the region $X$ by non-strict inequalities.
\end{remark}

Finally,

\begin{alg}[Computing the whole disguised toric locus]\label{algorithmDisguised3}
  \hfill \\
  \rm{Input:} A reaction network $G$.\\
  \rm{Output:} The disguised toric locus $\hat{K}(G)$ of $G$.\\

  Apply Algorithm~\ref{algorithmDisguised2} to $G$ and $\hat{G} = G_{\rm comp}$.
\end{alg}

\begin{remark}\label{rmk:}
  Note that using $G_{\rm comp}$ can be computationally unfeasible.\\
On the other hand, if in order to simplify our computations we choose $\rho$ (and $\hat G$) such that $\rho$ fails to be surjective or target-surjective, then we may still obtain {\em sufficient conditions} on ${\mathbf{k}}$  such that $(G, {\mathbf{k}})$ is disguised toric.
\end{remark}

\smallskip
\printbibliography

@article {MR3920470,
    AUTHOR = {Craciun, Gheorghe},
     TITLE = {Polynomial dynamical systems, reaction networks, and toric
              differential inclusions},
   JOURNAL = {SIAM J. Appl. Algebra Geom.},
  FJOURNAL = {SIAM Journal on Applied Algebra and Geometry},
    VOLUME = {3},
      YEAR = {2019},
    NUMBER = {1},
     PAGES = {87--106},
       URL = {https://doi.org/10.1137/17M1129076},
}

@article {CrPa,
    AUTHOR = {Craciun, Gheorghe and Pantea, Casian},
     TITLE = {Identifiability of chemical reaction networks},
   JOURNAL = {J. Math. Chem.},
  FJOURNAL = {Journal of Mathematical Chemistry},
    VOLUME = {44},
      YEAR = {2008},
    NUMBER = {1},
     PAGES = {244--259},
       URL = {https://doi.org/10.1007/s10910-007-9307-x},
}

@article {BD,
    AUTHOR = {Bihan, Fr\'{e}d\'{e}ric and Dickenstein, Alicia and Giaroli, Magali},
     TITLE = {Lower bounds for positive roots and regions of
              multistationarity in chemical reaction networks},
   JOURNAL = {J. Algebra},
  FJOURNAL = {Journal of Algebra},
    VOLUME = {542},
      YEAR = {2020},
     PAGES = {367--411},
       URL = {https://doi.org/10.1016/j.jalgebra.2019.10.002},
}

@book {Poin,
    AUTHOR = {Poincar\'{e}, Henri},
     TITLE = {The three-body problem and the equations of dynamics},
    SERIES = {Astrophysics and Space Science Library},
    VOLUME = {443},
      NOTE = {Poincar\'{e}'s foundational work on dynamical systems theory,
              Translated from the 1890 French original and with a preface by
              Bruce D. Popp},
 PUBLISHER = {Springer, Cham},
      YEAR = {2017},
     PAGES = {xxii+248},
       URL = {https://doi.org/10.1007/978-3-319-52899-1},
}

@article {MR2800059,
    AUTHOR = {Szederk\'{e}nyi, G\'{a}bor and Hangos, Katalin M.},
     TITLE = {Finding complex balanced and detailed balanced realizations of
              chemical reaction networks},
   JOURNAL = {J. Math. Chem.},
  FJOURNAL = {Journal of Mathematical Chemistry},
    VOLUME = {49},
      YEAR = {2011},
    NUMBER = {6},
     PAGES = {1163--1179},
       URL = {https://doi.org/10.1007/s10910-011-9804-9},
}

@article {lorenz,
    AUTHOR = {Lorenz, Edward N.},
     TITLE = {Deterministic nonperiodic flow},
   JOURNAL = {J. Atmospheric Sci.},
  FJOURNAL = {Journal of the Atmospheric Sciences},
    VOLUME = {20},
      YEAR = {1963},
    NUMBER = {2},
     PAGES = {130--141},
       URL = {https://doi.org/10.1175/1520-0469(1963)020<0130:DNF>2.0.CO;2},
}

@article {MR2561288,
    AUTHOR = {Craciun, Gheorghe and Dickenstein, Alicia and Shiu, Anne and
              Sturmfels, Bernd},
     TITLE = {Toric dynamical systems},
   JOURNAL = {J. Symbolic Comput.},
  FJOURNAL = {Journal of Symbolic Computation},
    VOLUME = {44},
      YEAR = {2009},
    NUMBER = {11},
     PAGES = {1551--1565},
       URL = {https://doi.org/10.1016/j.jsc.2008.08.006},
}

@article{GlAttrConj,
AUTHOR={Craciun, Gheorghe},
TITLE={Toric Differential Inclusions and a Proof of the Global Attractor Conjecture},
URL={https://arxiv.org/abs/1501.02860},
}

@article{Connected,
AUTHOR={Craciun, Gheorghe and Sorea, Miruna-Stefana},
TITLE={The structure of the moduli spaces of toric dynamical systems},
URL={https://arxiv.org/abs/2008.11468},
}

@article {cjy-singleT,
    AUTHOR = {Craciun, Gheorghe and Jin, Jiaxin and Yu, Polly Y.},
     TITLE = {Single-target networks},
   JOURNAL = {Discrete Contin. Dyn. Syst. Ser. B},
  FJOURNAL = {Discrete and Continuous Dynamical Systems. Series B. A Journal
              Bridging Mathematics and Sciences},
    VOLUME = {27},
      YEAR = {2022},
    NUMBER = {2},
     PAGES = {799--},
      ISSN = {1531-3492},
   MRCLASS = {92C42 (37N25 80A23)},
  MRNUMBER = {4361661},
       DOI = {10.3934/dcdsb.2021065},
       URL = {https://doi.org/10.3934/dcdsb.2021065},
}

@inproceedings {horn74,
    AUTHOR = {Horn, F.},
     TITLE = {The dynamics of open reaction systems},
 BOOKTITLE = {Mathematical aspects of chemical and biochemical problems and
              quantum chemistry ({P}roc. {SIAM}-{AMS} {S}ympos. {A}ppl.
              {M}ath., {N}ew {Y}ork, 1974)},
     PAGES = {125--137. SIAM-AMS Proceedings, Vol. VIII},
      YEAR = {1974},
}

@Misc{DF,
          author = {Dickenstein, Alicia and Feliu, Elisenda},
          title = {Algebraic Methods for Biochemical Reaction Networks},
          howpublished = {in progress}
        }

@book {bernd,
    AUTHOR = {Sturmfels, Bernd},
     TITLE = {Gr\"{o}bner bases and convex polytopes},
    SERIES = {University Lecture Series},
    VOLUME = {8},
 PUBLISHER = {American Mathematical Society, Providence, RI},
      YEAR = {1996},
     PAGES = {xii+162},
}

@book{BM,
  title={Invitation to nonlinear algebra},
  author={Micha{\l}ek, Mateusz and Sturmfels, Bernd},
  volume={211},
  year={2021},
  publisher={American Mathematical Soc.}
}

@incollection {MR3457596,
    AUTHOR = {Dickenstein, Alicia},
     TITLE = {Biochemical reaction networks: an invitation for algebraic
              geometers},
 BOOKTITLE = {Mathematical {C}ongress of the {A}mericas},
    SERIES = {Contemp. Math.},
    VOLUME = {656},
     PAGES = {65--83},
 PUBLISHER = {Amer. Math. Soc., Providence, RI},
      YEAR = {2016},
       URL = {https://doi.org/10.1090/conm/656/13076},
}

@article {cjy,
    AUTHOR = {Craciun, Gheorghe and Jin, Jiaxin and Yu, Polly Y.},
     TITLE = {An efficient characterization of complex-balanced,
              detailed-balanced, and weakly reversible systems},
   JOURNAL = {SIAM J. Appl. Math.},
  FJOURNAL = {SIAM Journal on Applied Mathematics},
    VOLUME = {80},
      YEAR = {2020},
    NUMBER = {1},
     PAGES = {183--205},
       URL = {https://doi.org/10.1137/19M1244494},
}

@book {shiuTh,
    AUTHOR = {Shiu, Anne},
     TITLE = {Algebraic methods for biochemical reaction network theory},
      NOTE = {Thesis (Ph.D.)--University of California, Berkeley},
 PUBLISHER = {ProQuest LLC, Ann Arbor, MI},
      YEAR = {2010},
     PAGES = {116},
}

@article {CNP,
    AUTHOR = {Craciun, Gheorghe and Nazarov, Fedor and Pantea, Casian},
     TITLE = {Persistence and permanence of mass-action and power-law
              dynamical systems},
   JOURNAL = {SIAM J. Appl. Math.},
  FJOURNAL = {SIAM Journal on Applied Mathematics},
    VOLUME = {73},
      YEAR = {2013},
    NUMBER = {1},
     PAGES = {305--329},
       URL = {https://doi.org/10.1137/100812355},
}

@book {BPR,
    AUTHOR = {Basu, Saugata and Pollack, Richard and Roy, Marie-Fran\c{c}oise},
     TITLE = {Algorithms in real algebraic geometry},
    SERIES = {Algorithms and Computation in Mathematics},
    VOLUME = {10},
   EDITION = {Second},
 PUBLISHER = {Springer-Verlag, Berlin},
      YEAR = {2006},
     PAGES = {x+662},
}

@article {cy,
 AUTHOR = {Yu, Polly Y. and Craciun, Gheorghe},
     TITLE = {Mathematical Analysis of Chemical Reaction Systems},
      JOURNAL={Israel Journal of Chemistry, 58, 2018},
      URL = {https://arxiv.org/abs/1805.10371},
}

@book {feinberg,
    AUTHOR = {Feinberg, Martin},
     TITLE = {Foundations of chemical reaction network theory},
    SERIES = {Applied Mathematical Sciences},
    VOLUME = {202},
 PUBLISHER = {Springer, Cham},
      YEAR = {2019},
     PAGES = {xxix+454},
}

@article {MR1898209,
    AUTHOR = {Ilyashenko, Yulij},
     TITLE = {Centennial history of {H}ilbert's 16th problem},
   JOURNAL = {Bull. Amer. Math. Soc. (N.S.)},
  FJOURNAL = {American Mathematical Society. Bulletin. New Series},
    VOLUME = {39},
      YEAR = {2002},
    NUMBER = {3},
     PAGES = {301--354},
       URL = {https://doi.org/10.1090/S0273-0979-02-00946-1},
}

@article {MR400923,
    AUTHOR = {Horn, F. and Jackson, R.},
     TITLE = {General mass action kinetics},
   JOURNAL = {Arch. Rational Mech. Anal.},
  FJOURNAL = {Archive for Rational Mechanics and Analysis},
    VOLUME = {47},
      YEAR = {1972},
     PAGES = {81--116},
       URL = {https://doi.org/10.1007/BF00251225},
}

@book{guckenheimer2013nonlinear,
  title={Nonlinear oscillations, dynamical systems, and bifurcations of vector fields},
  author={Guckenheimer, John and Holmes, Philip},
  volume={42},
  year={2013},
  publisher={Springer Science \& Business Media}
}

@book {strogatz2001nonlinear,
    AUTHOR = {Strogatz, Steven H.},
     TITLE = {Nonlinear dynamics and chaos},
   EDITION = {Second},
      NOTE = {With applications to physics, biology, chemistry, and
              engineering},
 PUBLISHER = {Westview Press, Boulder, CO},
      YEAR = {2015},
     PAGES = {xiii+513},
}

\bigskip \medskip

\noindent
\footnotesize

{\bf Authors:}

\smallskip

\noindent Laura Brustenga i Moncusí\\ University of Copenhagen, Denmark\\
\hfill {\tt brust@math.ku.dk}
\vspace{\baselineskip}

\noindent Gheorghe Craciun\\ University of Wisconsin-Madison, USA\\
\hfill {\tt craciun@math.wisc.edu}
\vspace{\baselineskip}

\noindent Miruna-\c Stefana Sorea\\ SISSA - Scuola Internazionale Superiore di Studi Avanzati, Trieste, Italy and RCMA Lucian Blaga University Sibiu, Romania\\
\hfill {\tt msorea@sissa.it}

\vskip5cm

\end{document}